\documentclass[reqno,a4paper]{amsart}
\usepackage{amssymb}
\usepackage{amsmath}
\newcommand{\half}{\frac{1}{2}}

\newcommand{\R}{\mathbb{R}}
\newcommand{\ti}{\tilde}

\begin{document} 
\newtheorem{prop}{Proposition}[section]
\newtheorem{Def}{Definition}[section]
\newtheorem{theorem}{Theorem}[section]
\newtheorem{lemma}{Lemma}[section]
 \newtheorem{Cor}{Corollary}[section]

\title[Unconditional well-posedness for MKG]{\bf Unconditional well-posedness below energy norm for the Maxwell-Klein-Gordon system}
\author[Hartmut Pecher]{
{\bf Hartmut Pecher}\\
Fakult\"at f\"ur  Mathematik und Naturwissenschaften\\
Bergische Universit\"at Wuppertal\\
Gau{\ss}str.  20\\
42119 Wuppertal\\
Germany\\
e-mail {\tt pecher@math.uni-wuppertal.de}}
\date{}

\begin{abstract}
The Maxwell-Klein-Gordon equation
$ \partial^{\alpha} F_{\alpha \beta}  = -Im(\phi \overline{D_{\beta} \phi}) $ , $
D^{\mu}D_{\mu} \phi = m^2 \phi $ , where
$F_{\alpha \beta} = \partial_{\alpha} A_{\beta} - \partial_{\beta} A_{\alpha}$, $D_{\mu} = \partial_{\mu} - iA_{\mu} $,
in the (3+1)-dimensional case is known to be unconditionally well-posed in energy space, i.e. well-posed in the natural solution space. This was proven by Klainerman-Machedon and Masmoudi-Nakanishi in Coulomb gauge and by Selberg-Tesfahun in Lorenz gauge. The main purpose of the present paper is to establish that for both gauges this also holds true for data $\phi(0)$ in Sobolev spaces $H^s$ with less regularity, i.e. $s < 1$, but $s$ sufficently close to $1$. This improves the (conditional) well-posedness results in both cases, i.e. uniqueness in smaller solution spaces of Bourgain-Klainerman-Machedon type, which were essentially known by Cuccagna, Selberg and the author for $s > \frac{3}{4}$ , and which in Coulomb gauge is also contained in the present paper. In fact, the proof consists in demonstrating that any solution in the natural solution space for some $s > s_0$ belongs to a Bourgain-Klainerman-Machedon space where uniqueness is known. Here $s_0 \approx 0.914$ in Coulomb gauge and $s_0 \approx 0.907$ in Lorenz gauge.
\end{abstract}
\maketitle
\renewcommand{\thefootnote}{\fnsymbol{footnote}}
\footnotetext{\hspace{-1.5em}{\it 2010 Mathematics Subject Classification:} 35Q40, 35L70
 }
\normalsize 
\setcounter{section}{0}

\section{Introduction}

\noindent Consider the Maxwell-Klein-Gordon equations
\begin{align}
\label{1}
\partial^{\alpha} F_{\alpha \beta} & = -Im(\phi \overline{D_{\beta} \phi}) \\
\label{2}
D^{\mu}D_{\mu} \phi &= m^2 \phi
\end{align}
in Minkowski space $\mathbb{R}^{1+3} = \mathbb{R}_t \times \mathbb{R}^3_x$ with metric $diag(-1,1,1,1)$. Greek indices run over $\{0,1,2,3\}$, Latin indices over $\{1,2,3\}$, and the usual summation convention is used.  Here $m > 0 $ and
$$ \phi: \mathbb{R}\times \mathbb{R}^3 \to \mathbb{C} \, , \, A_{\alpha}: \mathbb{R} \times \mathbb{R}^3 \to \mathbb{R} \, , \, F_{\alpha \beta} = \partial_{\alpha} A_{\beta} - \partial_{\beta} A_{\alpha} \, , \, D_{\mu} = \partial_{\mu} + iA_{\mu} \, . $$
$A_{\mu}$ are the gauge potentials, $F_{\mu \nu}$ is the electromagnetic field. We use the notation $\partial_{\mu} = \frac{\partial}{\partial x_{\mu}}$, where we write $(x^0,x^1,x^2,x^3)=(t,x^1,x^2,x^3)$ and also $\partial_0 = \partial_t$.

We have gauge freedom for the Cauchy problem, because the system is invariant under the gauge transformation 
$ \phi \longrightarrow \phi' = e^{i\chi} \phi \, , \, A_{\mu} \longrightarrow A_{\mu}' + \partial_{\mu} \chi$
for any $\chi: \R^{3+1} \to \R$ . The most common gauges are the Coulomb gauge $\partial^j A_j =0$ , the Lorenz gauge $\partial^{\nu} A_{\nu} =0$ and the temporal gauge $A_0 = 0$ .

We first make some historical remarks. 

In Coulomb gauge Klainerman and Machedon \cite{KM} proved global well-posedness in energy space and above, i.e. for (large) data $$ A_{\nu}(0) = a_{0\nu} \, , \, (\partial_t A)_{\nu} = a_{1\nu} \, , \, \phi(0) = \phi_0 \, , \, (\partial_t\phi)(0) = \phi_1 \, ,$$
where
$$\nabla  a_{0\nu} \in H^{s-1} \, , \, a_{1\nu}\ \in H^{s-1} \, , \, \phi_0 \in H^s \, , \, \phi_1 \in H^{s-1} $$
with $ s \ge 1$ . The main progress over earlier results by Eardley and Moncrief \cite{EM} for smooth data was their detection of a null condition for the nonlinearities. The uniqueness was proven under an additional assumption ("Conditional well-posedness"). Without this restriction the result also holds, as obtained by Masmoudi and Nakanishi \cite{MN} in the natural solution space
$$ \phi \in C^0([0,T],H^1) \cap C^1([0,T],L^2) \, , \, A \in C^0([0,T], \dot{H}^1) \cap C^1([0,T],L^2) \, . $$
This is usually called unconditional uniqueness. This improved an earlier result of Zhou \cite{Z} for a simplified model equation. Conditional local well-posedness for $s > \frac{3}{4}$ and small data was obtained by Cuccagna  \cite{C}. Keel-Roy-Tao \cite{KRT} proved conditional local well-posedness for $s > \frac{5}{6}$ and global well-posedness for $ s > \frac{\sqrt{3}}{2}$ for large data. Machedon and Sterbenz \cite{MS} proved conditional local well-posedness almost down to the scalar-critical regularity $s > \half$ and small data. In 4+1 dimensions the almost optimal conditional local well-posedness result ($s > 1$) for large data was proven by Selberg \cite{S}. In 2+1 dimensions Czubak and Picula \cite{CP} obtained conditional local well-posedness for $s > \half$.

In Lorenz gauge there is also a null condition present in part of the nonlinearities. This was detected by Selberg and Tesfahun \cite{ST}, who proved unconditional global well-posedness in energy space for large data. Conditional local well-posedness for $s > \frac{3}{4}$ was proven by the author \cite{P1}, who also considered the case of $n+1$ dimensions for $n \ge 2$, where conditional local well-posedness holds for $s > \half$ for $n=2$ and $ s > \frac{n}{2}-\frac{3}{4}$ for $ n \ge 3$ (cf. \cite{P2}).

In temporal gauge Tao  \cite{T} obtained local well-posedness for the more general Yang-Mills equations for $s > \frac{3}{4}$ and small data. His methods may be used also to study the large data problem for Maxwell-Klein-Gordon in 3+1 dimensions. This was carried out by Yuan \cite{Y} and Pecher \cite{P3}, who proved conditional local well-posedness for $s > \frac{3}{4}$ as well as conditional global well-posedness and unconditional global well-posedness in energy space, respectively.

The present paper has three parts, all of which address the Cauchy problem for Maxwell-Klein-Gordon in 3+1 dimensions.

In section 2  it is proven in Theorem \ref{Theorem1} that conditional local well-posedness in Coulomb gauge in Bourgain-Klainerman-Machedon spaces $X^{s,b}$ holds for $s > \frac{3}{4}$ and large data, thus removing the small data restriction in Cuccagna's paper \cite{C}. That this is possible was already remarked by Selberg (cf.\cite{S}, Remark 3), whose method  is one of the basic tools for our result. In 3+1 dimensions it is possible to rely completely on $X^{s,b}$-spaces. For the basic estimate for the null terms we may refer to Tao \cite{T}. Moreover we use the bilinear estimates  by d'Ancona-Foschi-Selberg \cite{AFS}. 

Sections 3 and 4, which are the essential elements of the paper, consider the issue of unconditional well-posedness, i.e. uniqueness in the natural solution spaces and not in smaller spaces where existence is typically obtained. 

In section 3 we consider the Coulomb gauge and prove unconditional local well-posedness below energy for $\phi(0) \in H^s$ with $s > \frac{1}{8}(25-\sqrt{313}) \approx 0.9135$. We demonstrate by an iteration process that any solution in the natural solution space posesses more and more regularity until it belongs to a space of $X^{s,b}$-type, where uniqueness holds by part 1. This does not require the use of any null conditions but only relies on Strichartz type estimates.

In section 4 we consider the Lorenz gauge and obtain the unconditional local well-posedness result for $s > 0.907$ (Theorem \ref{Theorem5}). The method is similar to section 3. We reduce the uniqueness issue to proving that any solution in the natural solution space belongs to a space where uniqueness holds by \cite{P1}. It is technically complicated by the fact that in this paper we had to consider data for the potential $A$ in homogeneous Sobolev spaces (and homogeneous parts of the solutions), giving rise to unpleasant small frequency issues. Again we make no use of the null conditions but rely completely on Strichartz type estimates iteratively improving the regularity of the solution.

We define the Bourgain-Klainerman-Machedon spaces $X^{s,b}_{\pm}$ as the completion of the Schwarz space $\mathcal{S}({\mathbb R}^{3+1})$ with respect to the norm
$$ \|u\|_{X^{s,b}_{\pm}} = \| \langle \xi \rangle^s \langle  \tau \pm |\xi| \rangle^b \widehat{u}(\tau,\xi) \|_{L^2_{\tau \xi}}  \, ,$$ 
where $ \langle \, \cdot \, \rangle = (1+|\, \cdot \,|^2)^{\frac{1}{2}}$ ,
and $X^{s,b}_{\pm}[0,T]$ as the space of the restrictions to $[0,T] \times \mathbb{R}^3$.

We also define the wave-Sobolev spaces $H^{s,b}$ as the completion of  $\mathcal{S}({\mathbb R}^{3+1})$ with respect to the norm
$$ \|u\|_{H^{s,b}} =  \| \langle \xi \rangle^s \langle  |\tau| - |\xi| \rangle^b \widehat{u}(\tau,\xi) \|_{L^2_{\tau \xi}}  $$
and $H^{s,b}[0,T]$ as the space of the restrictions to $[0,T] \times \mathbb{R}^3$.

Let $\Lambda^{\alpha}$ , $\Lambda^{\alpha}_m$ and $D^{\alpha}$  be the operators with Fourier symbols $\langle \xi \rangle^{\alpha}$, $( m^2 + |\xi|^2)^{\frac{\alpha}{2}}$ and
$|\xi|^{\alpha}$ , respectively.

$\Box = \partial_{\mu} \partial^{\mu}$ is the d'Alembert operator, $a{\pm} = a\pm\epsilon$ for a sufficiently small $\epsilon > 0$ . 

The following variant of the Strichartz estimate is proven in \cite{ST}, Lemma 7.1.
\begin{prop}
\label{Lemma7.1}
Suppose $2 < q \le \infty$ and $2 \le r < \infty$ satisfy $\half \le \frac{1}{q} + \frac{1}{r} \le 1$ . Then the following estimate holds:
$$ \|u\|_{L^q_t L^r_x} \lesssim \| D^{1-\frac{2}{r}}u\|_{H^{0,1-(\frac{1}{q} + \frac{1}{r})+}} \, , $$
and (by duality)
$$  \| D^{\frac{2}{r}-1}u\|_{H^{0,(\frac{1}{q} + \frac{1}{r})-1-}}\lesssim \|u\|_{L^{q'}_t L^{r'}_x} \, , $$
where $\frac{1}{q}+\frac{1}{q'} = 1 $ and $\frac{1}{r}+\frac{1}{r'} = 1$ .
\end{prop}
Moreover we need the following bilinear refinement of Strichartz type estimates given by \cite{AFS}, where many limit cases are included which we do not need.

\begin{theorem}
\label{Theorem3.2}
The estimate
$$ \|uv\|_{H^{-s_0,-b_0}} \lesssim \|u\|_{H^{s_1,b_1}} \|v\|_{H^{s_2,b_2}} $$
holds, provided the following conditions are satisfied:
\begin{align*}
b_0+b_1+b_2 > \half \, , \, b_0+b_1 &> 0 \, , \, b_0+b_1 > 0 \, , \, b_1+b_2 > 0 \\
s_0+s_1+s_2 & > 2-(b_0+b_1+b_2) \\
s_0+s_1+s_2 & > \frac{3}{2} -(b_0+b_1) \\
s_0+s_1+s_2 & > \frac{3}{2} -(b_0+b_2) \\
s_0+s_1+s_2 & > \frac{3}{2} -(b_1+b_2)
\end{align*} 
\begin{align*}
s_0+s_1+s_2 & > 1-b_0 \\
s_0+s_1+s_2 & > 1-b_1 \\
s_0+s_1+s_2 & > 1-b_2 \\
s_0+s_1+s_2 & > 1 \\
(s_0+s_1+s_2)+(s_1+s_2+b_0) & > \frac{3}{2} \\
(s_0+s_1+s_2)+(s_0+s_2+b_1) & > \frac{3}{2} \\
(s_0+s_1+s_2)+(s_0+s_1+b_2) & > \frac{3}{2} \\
s_1+s_2 > -b_0 \, , \, s_0+s_2 &> -b_1 \, , \, s_0+s_1 > -b_2 \\
s_1+s_2 \ge 0 \, , \, s_0+s_2 &\ge 0 \, , \, s_0+s_1 \ge 0 \, .
\end{align*}
\end{theorem}
The Sobolev multiplication law is standard (cf. \cite{T}, Cor. 3.16):
\begin{prop}
\label{Prop.4.4}
Let $s_0,s_1,s_2 \in\R$ . \\
a. If $s_0+s_1+s_2 >  \frac{3}{2}$ , $s_0+s_1 \ge 0$ , $s_0+s_2 \ge 0$ , $s_1+s_2 \ge 0$ , 
then
$$ \|uv\|_{H^{-s_0}} \lesssim \|u\|_{H^{s_1}}  \|v\|_{H^{s_2}} \, . $$
b.  If $s_0+s_1+s_2  = \frac{3}{2}$ , $s_0+s_1 > 0$ , $s_0+s_2 > 0$ , $s_1+s_2 > 0$ , 
then
$$ \|uv\|_{\dot{H}^{-s_0}} \lesssim \|u\|_{\dot{H}^{s_1}}  \|v\|_{\dot{H}^{s_2}} \, . $$
\end{prop}

\section{Conditional local well-posedness in Coulomb gauge}
\noindent If we consider the Maxwell-Klein-Gordon equations (\ref{1}),(\ref{2}) in Coulomb gauge
\begin{equation}
\label{a}
\partial^j A_j =0 \, ,
\end{equation}
we obtain
\begin{align}
\label{b}
\square A_j & = - Im(\phi \overline{\partial_j \phi}) + |\phi|^2 A_j - \partial_j \partial_t A_0\\
\label{c}
(\square -m^2) \phi & = -2i A^j \partial_j \phi +2iA_0 \partial_t \phi + i (\partial_t A_0)\phi + A^{\mu} A_{\mu} \phi \\
\label{d'}
\Delta A_0 & = -Im(\phi \overline{\partial_t \phi}) + |\phi|^2 A_0 \, .
\end{align}
We want to solve (\ref{a})-(\ref{d'}) and the following initial conditions:
\begin{equation}
\label{IC}
 A(0) = a_0 \, , \, (\partial_t A)(0) = a_1 \, , \, \phi(0) = \phi_0 \, , \, (\partial_t \phi)(0) = \phi_1 \, , 
\end{equation}
where $A=(A_1,A_2,A_3)$ . If the Coulomb condition (\ref{a}) is imposed, we necessarily have to require the compatibility condition
\begin{equation}
\label{CC}
 \partial^j a_{0j} = \partial^j a_{1j} = 0 \, . 
\end{equation}
We add the following equation to the system by differentiating (\ref{d'}) to $t$ and using (\ref{a}).
\begin{equation}
\label{e'}
\Delta B_0 = - Im \, \partial^j(\phi \overline{\partial_j \phi}) + \partial^j(|\phi|^2 A_j) \, ,
\end{equation}
where $B_0 = \partial_t A_0$ . It is well-known (cf. \cite{S}), that (\ref{b}),(\ref{c}),(\ref{d'}),(\ref{e'}) can be written in the following form:
\begin{align}
\nonumber
(\square -1)A_j & = 2 R^k (-\Delta)^{-\half} Q_{jk}(Re \, \phi,Im \, \phi) + P(|\phi|^2A_j) -A_j \\
\label{f}
& =: M_j(A_j,\phi) \\
\nonumber
(\square -m^2) \phi & = -i Q_{jk}(\phi, (-\Delta)^{-\half} (R^j A^k - R^k A^j)) + 2i A_0 \partial_t \phi + i B_0 \phi + A^{\mu} A_{\mu} \phi \\
\label{g}
& =: N(A,\phi) \\
\label{d}
\Delta A_0 & = -Im(\phi \overline{\partial_t \phi}) + |\phi|^2 A_0 \\
\label{e}
\Delta B_0 &= - Im \, \partial^j(\phi \overline{\partial_j \phi}) + \partial^j(|\phi|^2 A_j) \, .
\end{align}
Here $Q_{jk}(u,v) := \partial_j u \partial_k v - \partial_k u \partial_j v$ are the standard null forms. $P$ denotes the projection onto the divergence-free vector fields defined by $PA = \Delta^{-1} \nabla \times (\nabla \times A)$,  and $R_j := D^{-1} \partial_j$ is the Riesz transform. Both operators are bounded in all the spaces considered in the sequel.

Our first aim is to solve the elliptic equations (\ref{d}),(\ref{e}) in order to obtain $A_0=A_0(\phi)$ and $B_0=B_0(\phi,A)$ for given sufficiently regular $A$ and $\phi$ , so that we are left with the purely hyperbolic system (\ref{f}),(\ref{g}). We want to obtain a solution in the following regularity class:
\begin{align*}
A_j,\phi \in H^{s,\frac{3}{4}+}[0,T] \, &, \, \partial_t A_j,\partial_t \phi \in H^{s-1,\frac{3}{4}+}[0,T] \\
A_0 \in C^0([0,T],\dot{H}^1 ) \, &, \, \nabla A_0 \in L^2([0,T],H^{\half+}) \\
 B_0 \in C^0([0,T],L^2 ) \, &, \, \nabla B_0 \in L^2([0,T],H^{-\half+})
\end{align*}
under the assumption $s > \frac{3}{4}$ .

We first consider the elliptic equation (\ref{d}).

\begin{lemma}
\label{Lemma6}
(cf. \cite{S}, Lemma 6) 
The equation
$$ \Delta u - |\phi|^2 u = -Im(\phi f) $$
has a unique solution $u \in \dot{H}^1$ , which fulfills
\begin{equation}
\label{76}
\| u \|_{\dot{H}^1} \lesssim \|\phi\|_{H^{\frac{3}{4}}} \|f\|_{H^{-\frac{1}{4}}} \, , 
\end{equation}
provided $\phi \in H^{\frac{3}{4}}$ , $f \in H^{-\frac{1}{4}}$ .
\end{lemma}
\begin{proof}
We have to solve the equation
\begin{equation}
\label{77}
\int (\nabla u \cdot \nabla v + |\phi|^2 uv)\, dx = \int Im(\phi f)v \, dx \, .
\end{equation}
Now
\begin{align}
\label{78}
| \int |\phi|^2 uv \, dx | & \le \|\phi\|_{L^3}^2 \|u\|_{L^6} \|v\|_{L^6} \lesssim \|\phi\|_{H^{\frac{3}{4}}}^2 \|u\|_{\dot{H}^1} \|v\|_{\dot{H}^1}  \\
\label{79}
| \int Im(\phi f) v \, dx | & \le \|\phi v\|_{H^{\frac{1}{4}}} \|f\|_{H^{-\frac{1}{4}}} \lesssim \|\phi\|_{H^{\frac{3}{4}}}  \|v\|_{\dot{H}^1} \|f\|_{\dot{H}^{-\frac{1}{4}}}
\end{align}
by Prop. \ref{Prop.4.4}. For $v = \overline{u} $ we obtain
$$ \|u\|_{\dot{H}^1}^2 + \|\phi u\|_{L^2}^2 \le \half \|u\|_{\dot{H}^1}^2 + \frac{c}{2} \|\phi\|_{H^{\frac{3}{4}}}^2 \|f\|_{H^{-\frac{1}{4}}}^2 \, , $$
thus we obtain (\ref{76}). By (\ref{78}) the left hand side of (\ref{77}) defines a scalar product  on $\dot{H}^1$ with a norm which is equivalent to the $\dot{H}^1$ - norm. By (\ref{79}) the right hand side of (\ref{77}) is a bounded linear functional on $\dot{H}^1$, so that the Riesz representation theorem gives the existence of a unique solution.
\end{proof}
It is easy to generalize Lemma \ref{Lemma6} to
\begin{lemma}
\label{Lemma7} (cf. \cite{S}, Lemma 7)
Let $\phi,\psi \in H^{\frac{3}{4}}$ and $f,g \in H^{-\frac{1}{4}}$ . Let $u,v$ be the solutions of
\begin{align*}
\Delta u - |\phi|^2 u & = - Im(\phi f) \\
\Delta v - |\psi|^2 v & = - Im(\psi g)   \, .
\end{align*}
Then
\begin{align*}
\|u-v\|_{\dot{H}^1} \lesssim (\|\phi\|_{H^{\frac{3}{4}}} + \|\psi\|_{H^{\frac{3}{4}}} +\|g\|_{H^{-\frac{1}{4}}})(\|\phi - \psi\|_{H^{\frac{3}{4}}} + \|f-g\|_{H^{-\frac{1}{4}}}) \, .
\end{align*}
\end{lemma}
An immediate consequence is the following lemma.

\begin{lemma}
\label{Lemma8}
(cf. \cite{S},Lemma 8) Assume $\phi \in C^0([0,T],H^{\frac{3}{4}}) \cap C^1([0,T],H^{-\frac {1}{4}})$ and $f \in  C^0([0,T],H^{-\frac{1}{4}})$ . Then the equation (\ref{d}) has a unique solution $A_0 = A_0(\phi) \in C^0([0,T],\dot{H}^1)$ and
\begin{align*}
\|A_0(t)-A_0(s)\|_{\dot{H}^1} \lesssim (\|\phi\|_{L^{\infty}_t H^{\frac{3}{4}}_x} + \|\partial_t \phi\|_{L^{\infty}_t H^{-\frac{1}{4}}_x})& (\|\phi(t)-\phi(s)\|_{H^{\frac{3}{4}}} \\ 
&+ \|\partial_t\phi(t)-\partial_t\phi(s)\|_{H^{-\frac{1}{4}}}) \, .
\end{align*}
\end{lemma}
Moreover we obtain
\begin{lemma}
\label{Lemma9}
Assume $\phi,\psi \in C^0([0,T],H^{\frac{3}{4}}) \cap C^1([0,T],H^{-\frac{1}{4}})$ . Let $A_0(\phi)$ , $A_0(\psi)$ be the solutions of
\begin{align}
\label{80}
\Delta A_0(\phi) & = Im(\phi \overline{\partial_t \phi}) + |\phi|^2 A_0(\phi) \, , \\
\Delta A_0(\psi) & = Im(\psi \overline{\partial_t \psi}) + |\psi|^2 A_0(\psi) \, .
\label{81}
\end{align}
Then
\begin{align*}
\| A_0(\phi) -A_0(\psi)\|_{L^{\infty}_t \dot{H}^1_x} & \lesssim (\|\phi\|_{L^{\infty}_t H^{\frac{3}{4}}_x} + \|\psi\|_{L^{\infty}_t H^{\frac{3}{4}}_x} +\|\partial_t \phi\|_{L^{\infty}_t H^{-\frac{1}{4}}_x}  +\|\partial_t \psi\|_{L^{\infty}_t H^{-\frac{1}{4}}_x}) \\
& \quad \cdot ((\|\phi - \psi\|_{L^{\infty}_t H^{\frac{3}{4}}_x} +  \|\partial_t \psi - \partial_t \psi\|_{L^{\infty}_t H^{-\frac{1}{4}}_x}) \, .
\end{align*}
\end{lemma}
\begin{proof}
Use Lemma \ref{Lemma7} for each fixed $t$ and $f= \overline{\partial_t \phi (t)}$ , $g=\overline{\partial_t \psi (t)}$ and take the supremum over $t \in [0,T]$ .
\end{proof}

\begin{lemma}
\label{Lemma10}
Let $\phi,\psi \in H^{s,\half +}$ , $\partial_t \phi , \partial_t \phi \in H^{s-1,\half+}$ for $s > \frac{3}{4}$ , and $A_0(\phi)$ , $A_0(\psi)$ the solutions of (\ref{80}),(\ref{81}). Then
\begin{align}
\nonumber
&\|\Delta( A_0(\phi) -A_0(\psi))\|_{H^{-\half+,0}}  \lesssim (\|\phi\|_{H^{s,\half+}} + \|\psi\|_{H^{s,\half+}} \\
& \quad \quad +\|\partial_t \phi\|_{H^{s-1,\half+}} +\|\partial_t \psi\|_{H^{s-1,\half+}} ) 
\label{82}
 \cdot ((\|\phi - \psi\|_{H^{s,\half+}}  +  \|\partial_t \psi - \partial_t \psi\|_{H^{s-1,\half+}} ) \, .
\end{align}
\end{lemma}
\begin{proof}
The difference of the first terms on the right hand sides of (\ref{80}) and (\ref{81}) is easily estimated by using
$$ \|uv\|_{H^{-\half+,0}} \lesssim \|u\|_{H^{s,\half+}} \|v\|_{H^{s-1,\half+}} \, , $$
which holds by Theorem \ref{Theorem3.2} for $s > \frac{3}{4}$  . For the difference of the cubic terms we obtain by Prop. \ref{Prop.4.4} and Lemma \ref{Lemma8} :
\begin{align*}
\|(\phi - \psi) \overline{\phi} A_0(\phi)\|_{H^{-\half+,0}} & \lesssim \|A_0(\phi)\|_{L^2_t \dot{H}^1_x} \|(\phi - \psi) \overline{\phi} \|_{L^{\infty}_t H^{0+}_x} \\
& \lesssim \|A_0(\phi)\|_{L^2_t \dot{H}^1_x} \|\phi - \psi \|_{L^{\infty}_t H^s_x} \|\phi \|_{L^{\infty}_t H^s_x} \\
& \lesssim (\|\phi\|_{H^{s,\half+}} + \|\partial_t \phi\|_{H^{s-1,\half+}})^2 \|\phi - \psi \|_{H^{s,\half+}} \|\phi \|_{H^{s,\half+}} \, .
\end{align*}
Moreover similarly
$$ \| |\phi|^2(A_0(\phi)-A_0(\psi))\|_{H^{-\half+,0}} \lesssim \|\phi\|_{H^{s,\half+}}^2 \|A_0(\phi) -A_0(\psi)\|_{L^2_t \dot{H}^1_x} \, , $$
and Lemma \ref{Lemma9} gives the result.
\end{proof}
Next we consider the elliptic equation (\ref{e}).
\begin{lemma}
\label{Lemma11}
Let $ s >\frac{3}{4}$ , $\phi, A \in C^0([0,T],H^s)$. The solution $B_0 = B_0(\phi,A)$ of (\ref{e}) belongs to $C^0([0,T],L^2_x)$ . Moreover $D B_0 \in H^{-\half+,0}$ and
\begin{align}
\nonumber
&\| D(B_0(\phi,A) - B_0(\psi,A')\|_{H^{-\half+,0}} \\
\nonumber
 &\lesssim c(\|\phi\|_{H^{s,\half+}} , \psi\|_{H^{s,\half+}} , \|\partial_t \phi\|_{H^{s-1,\half+}} , \| \partial_t \psi\|_{H^{s-1,\half+}} , \|A\|_{H^{s,\half+}} , \|A'\|_{H^{s,\half+}}) \\
& \quad \quad \cdot(\|\phi - \psi \|_{H^{s,\half+}} +\|\partial_t(\phi - \psi) \|_{H^{s-1,\half+}} + \|A - A' \|_{H^{s,\half+}}) \, ,
\label{83}
\end{align}
where $c$ is a continuous function.
\end{lemma}
\begin{proof}
We obtain
\begin{align*}
&\|B_0(t) - B_0(s)\|_{L^2_x} \\
 & \quad \lesssim \sum_j (\|\phi(t) \overline{\partial_j \phi(t)} - \phi(s) \overline{\partial_j \phi(s)} \|_{\dot{H}^{-1}_x} + \| |\phi(t)|^2 A_j(t) - |\phi(s)|^2 A_j(s)\|_{\dot{H}^{-1}}) \, .
\end{align*}
By Prop. \ref{Prop.4.4} we estimate the first term on the right hand side by
$$ \sum_j(\| \phi(t)-\phi(s) \|_{H^{\frac{3}{4}}} \|\partial_j \phi(t)\|_{\dot{H}^{-\frac{1}{4}}} + \|\phi(t)\|_{H^{\frac{3}{4}}} \|\partial_j \phi(t) - \partial_j \phi(s)\|_{\dot{H}_x^{-\frac{1}{4}}}) $$
and the second term by
\begin{align*}
&\sum_j ( \| |\phi(t)|^2 - |\phi(s)|^2 \|_{\dot{H}^{-\frac{1}{4}}} \|A_j(t)\|_{H^{\frac{3}{4}}} + \| |\phi(s)|^2 \|_{\dot{H}^{-\frac{1}{4}}} \|A_j(t)-A_j(s)\|_{H^{\frac{3}{4}}}) \\
& \quad\lesssim \sum_j \| \phi(t) - \phi(s) \|_{H^{\frac{3}{4}}} (\|\phi(t)\|_{H^{\half}} + \|\phi(s)\|_{H^{\half}}) \|A_j(t)\|_{H^{\frac{3}{4}}} \\
& \hspace{10em}+\| \phi(s)\|_{H^{\frac{5}{8}}}^2 \|A_j(t)-A_j(s)\|_{H^{\frac{3}{4}}}) \, .
\end{align*}
This implies $B_0 \in C^0([0,T],L^2)$ . Because the terms which we have to estimate are multilinear we reduce (\ref{83}) to estimating
$$ \| D B_0\|_{H^{-\half+,0}} \lesssim \sum_j(\|\phi \partial_j \phi\|_{H^{-\half+,0}} + \| |\phi|^2 A_j\|_{H^{-\half+,0}}) \, . $$
The first term is bounded by $$ \|\phi\|_{H^{s,\half+}} \|\nabla \phi\|_{H^{s-1,\half+}} \lesssim \|\phi\|_{H^{s,\half+}}^2 $$ by Theorem \ref{Theorem3.2}. The second term is estimated as follows
\begin{align*}
\| |\phi|^2 A_j\|_{H^{-\half+,0}} & \lesssim \| |\phi|^2 \|_{H^{-\frac{1}{4},\half+}} \|A_j\|_{H^{s,\half+}} 
 \lesssim \| \phi\|_{H^{s,\half+}}^2  \|A_j\|_{H^{s,\half+}} \, ,
\end{align*}
where we used Theorem \ref{Theorem3.2} twice with parameters $s_0 = \half-$ , $s_1=-\frac{1}{4}$ , $s_2=s$, $b_0=0$ , $b_1=b_2=\half+$ for the first step, so that $s_0+s_1+s_2 = s+\frac{1}{4}- > 1$ and $s_1+s_2= s-\frac{1}{4} > \half$ , whereas for the second step we have $s_0 = \frac{1}{4}$ , $s_1=s_2=s$ , $b_0 = -\half-$ , $b_1=b_2=\half+$ , so that $s_0+s_1+s_2 = 2s+\frac{1}{4} > \frac{7}{4}$ , $(s_0+s_1+s_2)+s_1+s_2+ b_0 > \frac{7}{4} + \frac{3}{2}-\half > \frac{3}{2}$ , thus the conditions of Theorem \ref{Theorem3.2} are satisfied.
\end{proof}
Next we consider the equation 
\begin{equation}
\label{83'}
\Delta u -|\phi|^2 u = f \, .
\end{equation}

\begin{lemma}
\label{Lemma12}
(cf. \cite{S}, Lemma 8)
If $\phi \in H^{\frac{3}{4}}$ and $f \in L^{\frac{6}{5}}$ , (\ref{83'}) has a unique solution $u \in \dot{H}^1$ and 
\begin{equation}
\label{84}
\|u\|_{\dot{H}^1} \lesssim \|f\|_{L^{\frac{6}{5}}} \, . 
\end{equation}
If $\phi,\psi \in H^{\frac{3}{4}}$ and $f,g \in L^{\frac{6}{5}}$, $\Delta u - |\phi|^2 u = f$ and $\Delta v - |\psi|^2 v = g$ , then
$$ \|u-v\|_{\dot{H}^1} \lesssim (\|\phi\|_{H^{\frac{3}{4}}} + \|\psi\|_{H^{\frac{3}{4}}}) \|g\|_{L^{\frac{6}{5}}} \|\phi - \psi\|_{H^{\frac{3}{4}}} + \|f-g\|_{L^{\frac{6}{5}}} \, . $$
\end{lemma}
\begin{proof}
Replacing $-Im(\phi f)$ by $f$ , we obtain as in Lemma \ref{Lemma6} :
$$ | \int fv \, dx | \le \|f\|_{L^{\frac{6}{5}}} \|v\|_{L^6} \le \half \|v\|_{\dot{H}^1}^2 + c \|f\|_{L^{\frac {6}{5}}}^2 \, . $$
This gives the claimed result for $u$ as in Lemma \ref{Lemma6}. Moreover
$$ (\Delta - |\phi|^2)(u-v) = (|\phi|^2 - |\psi|^2) v + (f-g) \, . $$
By (\ref{84}) we obtain
\begin{align*}
\|u-v\|_{\dot{H}^1} & \lesssim \| |\phi + \psi| |\phi - \psi| v \|_{L^{\frac{6}{5}}} + \|f-g\|_{L^{\frac{6}{5}}} \\
& \lesssim \|\phi + \psi\|_{L^3} \| \phi - \psi\|_{L^3} \| v\|_{L^6} + \|f-g\|_{L^{\frac{6}{5}}} \\
& \lesssim (\|\phi\|_{H^{\frac{3}{4}}} + \|\psi\|_{H^{\frac{3}{4}}}) \|\phi - \psi\|_{H^{\frac{3}{4}}} \|g\|_{L^{\frac{6}{5}}} + \|f-g\|_{L^{\frac{6}{5}}} \, .
\end{align*}
\end{proof}
This lemma is applied in order to prove

\begin{lemma}
\label{Lemma13}
If $\phi \in C^0([0,T],H^{\frac{3}{4}}) \cap C^1([0,T],H^{\frac{1}{4}}) \cap C^2([0,T],L^2)$, and if $A_0 \in C^0([0,T],\dot{H}^1)$ is the solution of (\ref{d}), we obtain $\partial_t A_0 \in C^0([0,T],\dot{H}^1)$ .
\end{lemma}
\begin{proof}
By (\ref{d}) we know
$$ \Delta(\partial_t A_0) - |\phi|^2(\partial_t A_0) = \partial_t(|\phi|^2) A_0 - Im(\phi \overline{\partial_t^2 \phi}) = : f \, . $$
This step is justified in the sequel, because the right hand side is proven to belong to $C^0([0,T],L^{\frac{6}{5}})$. By Lemma \ref{Lemma12} with $u=u(t)$ and $v=u(s)$ we obtain
\begin{align*}
&\| (\partial_t A_0)(t) - (\partial_t A_0)(s)\|_{\dot{H}^1} \\
&\lesssim \|\phi\|_{L^{\infty}_t (H_x^{\frac{3}{4}})} \|f\|_{L^{\infty}_t (L_x^{\frac{6}{5}})} \|\phi(t)-\phi(s)\|_{H^{\frac{3}{4}}} + \|f(t)-f(s)\|_{L^{\frac{6}{5}}} \, .
\end{align*}
Next we estimate by Sobolev
\begin{align*}
&\|f(t)-f(s)\|_{L^{\frac{6}{5}}}\\
 & \lesssim \| \partial_t(|\phi(t)|^2) A_0(t) - \partial_t(|\phi(s)|^2) A_0(s)\|_{L^{\frac{6}{5}}} + \| \phi(t) (\partial_t^2 \phi)(t) -  \phi(s) (\partial_t^2 \phi)(s)\|_{L^{\frac{6}{5}}} \\
& \lesssim \left( (\| \phi(t)\|_{L^4} + \| \phi(s)\|_{L^4}) \|\partial_t \phi(t) - \partial_t \phi(s) \|_{L^\frac{12}{5}} + \| \phi(t) - \phi(s)\|_{L^4} \|\partial_t \phi(t) \|_{L^\frac{12}{5}} \right) \\
& \quad \quad \cdot (\|A_0(t)\|_{L^6} + \|A_0(s)\|_{L^6}) 
 + \|\phi\|_{L^4} \|\partial_t \psi(t)\|_{L^{\frac{12}{5}}} \|A_0(t)-A_0(s)\|_{L^6} \\
& \,+ \|\phi(t)-\phi(s)\|_{L^3} \|(\partial_t^2 \phi)(t)\|_{L^2} 
 + (\| \phi(t)\|_{L^3} + \|\phi(s)\|_{L^3}) \|(\partial_t^2 \phi)(t) - (\partial_t^2 \phi)(s)\|_{L^2} \\
& \lesssim \left( (\| \phi(t)\|_{H^{\frac{3}{4}}} + \| \phi(s)\|_{H^{\frac{3}{4}}}) \|\partial_t \phi(t) - \partial_t \phi(s) \|_{H^{\frac{1}{4}}} + \| \phi(t) - \phi(s)\|_{H^{\frac{3}{4}}} \|\partial_t \phi(t) \|_{H^{\frac{1}{4}}} \right) \\
& \quad \quad(\|A_0(t)\|_{\dot{H}^1} + \|A_0(s)\|_{\dot{H}^1}) 
 + \|\phi\|_{H^{\frac{3}{4}}} \|\partial_t \phi(t)\|_{{H^{\frac{1}{4}}}} \|A_0(t)-A_0(s)\|_{\dot{H}^1} \\
&+ \|\phi(t)-\phi(s)\|_{H^{\frac{3}{4}}} \|(\partial_t^2 \phi)(t)\|_{L^2} 
 + (\| \phi(t)\|_{H^{\frac{3}{4}}} \hspace{-0.2em}
+ \hspace{-0.2em}\|\phi(s)\|_{H^{\frac{3}{4}}}) \|(\partial_t^2 \phi)(t) - (\partial_t^2 \phi)(s)\|_{L^2} 
\end{align*}
which implies the claimed result.
\end{proof}

For given $A_j,\phi \in H^{s,\frac{3}{4}+}$ we denote by $A_0(\phi)$ and $B_0(\phi,A)$ the solution of (\ref{d}) and (\ref{e}), respectively, which we insert in (\ref{g}), and prove the following local well-posedness result for (\ref{f}),(\ref{g}),(\ref{d}),(\ref{e}).

\begin{theorem}
\label{Theorem1}
Let $s >\frac{3}{4}$ . The system (\ref{f}),(\ref{g}),(\ref{d}),(\ref{e}) with initial conditions (\ref{IC}) fulfilling (\ref{CC}) is locally well-posed for initial data $a_{0j},\phi_0 \in H^s$ , $a_{1j},\phi_1 \in H^{s-1}$ ($j=1,2,3$) in the space
$$\phi,A_j \in H^{s,\frac{3}{4}+}[0,T] \quad , \quad \partial_t \phi, \partial_t A_j \in H^{s-1,\frac{3}{4}+}[0,T] \, . $$
Moreover
$$ A_0 \in C^0([0,T],\dot{H}^1) \quad , \quad \Delta A_0 \in L^2([0,T],H^{-\half}) \, , $$
$$ B_0 \in C^0([0,T],L^2) \quad ,  \, \quad D B_0 \in L^2([0,T],H^{-\half}) \, . $$
The solution depends continuously on the data, persistence of higher regularity holds, and $A,\phi \in C^{\infty}([0,T] \times \R^3)$ , if the data belong to $H^k$ for all $k$.
\end{theorem}
\noindent {\bf Remark:} $A_j,\phi \in C^0([0,T],H^s) \cap C^1([0,T],H^{s-1})$. \\[1em]
Next we refer to Selberg \cite{S}, who proved that this theorem in connection with the uniqueness result in Lemma \ref{Lemma12} implies $B_0 = \partial_t A_0$ , provided for sufficiently regular data we obtain $B_0,\partial_t A_0 \in C^0([0,T],\dot{H}^1)$. For $\partial_t A_0$ this was established in Lemma \ref{Lemma13}. For $B_0$ we estimate by Prop. \ref{Prop.4.4} crudely
$$\|B_0\|_{\dot{H}^1} \lesssim \sum_j (\|\phi \partial_j \phi\|_{L^2}+ \| |\phi|^2 A_j\|_{L^2}) \lesssim \|\phi\|_{H^2}^2 (1+ \sum_j \|A_j\|_{H^1}) $$
and use $\phi \in C([0,T],H^2)$ , $A_j \in C^0([0,T],H^1)$ by persistence of higher regularity. Once this has been obtained it is easy to see that the systems (\ref{a}),(\ref{b}),(\ref{c}),(\ref{d'}) and (\ref{f}),(\ref{g}),(\ref{d}),(\ref{e}) are equivalent.\\[1em]
Thus we conclude

\begin{theorem}
\label{Theorem2}
Let $s > \frac{3}{4}$ . The Maxwell-Klein-Gordon system (\ref{1}),(\ref{2}) in Coulomb gauge (\ref{a}) is locally well-posed for initial data $a_{0j},\phi_0 \in H^s$ , $a_{1j},\phi_1 \in H^{s-1}$ ($j=1,2,3$) , which fulfill (\ref{CC}) , in the sense of Theorem \ref{Theorem1}.
\end{theorem}
\begin{proof}[Proof of Theorem \ref{Theorem1}]
By well-known arguments the proof reduces to the following nonlinear estimate (cf. (\ref{f}),(\ref{g})):
\begin{align}
\nonumber
&\| M_j(A_j,\phi) - M_j(A_j',\phi')\|_{H^{s-1,-\frac{1}{4}++}} + \|N(A,\phi) - N(A',\phi')\|_{H^{s-1,-\frac{1}{4}++}}      \\
\nonumber
 & \lesssim c(\|A_j\|_{H^{s,\frac{3}{4}+}} , \|A_j'\|_{H^{s,\frac{3}{4}+}} , \|\phi\|_{H^{s,\frac{3}{4}+}} , \|\phi'\|_{H^{s,\frac{3}{4}+}}) \\
\label{85}
&\quad \quad \cdot (\|A_j - A_j'\|_{H^{s,\frac{3}{4}+}} + \|\phi - \phi'\|_{H^{s,\frac{3}{4}+}}) \, ,
\end{align}
where $c$ denotes a continuous function.
All these terms are multilinear so that we only have to prove the following estimates (\ref{86})-(\ref{94}). 
\begin{align}
\label{86}
\|D^{-1} Q_{jk}(u,v)\|_{H^{s-1,-\frac{1}{4}++}} & \lesssim \|u\|_{H^{s,\frac{3}{4}+}} \|v\|_{H^{s,\frac{3}{4}+}} \\
\label{87}
\|Q_{jk}(u,D^{-1}v)\|_{H^{s-1,-\frac{1}{4}++}} & \lesssim \|u\|_{H^{s,\frac{3}{4}+}} \|v\|_{H^{s,\frac{3}{4}+}} \\
\label{88}
\| uvw \|_{H^{s-1,-\frac{1}{4}++}} & \lesssim \|u\|_{H^{s,\frac{3}{4}+}} \|v\|_{H^{s,\frac{3}{4}+}} \|w\|_{H^{s,\frac{3}{4}+}}  \\
\label{89}
\|A_0 \partial_t \phi\|_{H^{s-1,0}} & \lesssim \|DA_0\|_{H^{\half+,0}} \|\partial_t \phi\|_{H^{s-1,\half+}} \\
\label{90}
\|B_0  \phi\|_{H^{s-1,0}} & \lesssim \|DB_0\|_{H^{-\half+,0}} \|\phi\|_{H^{s,\half+}} \\
\label {91}
\|A_0 A_0 \phi\|_{H^{s-1,0}} & \lesssim \|A_0\|_{L^{\infty}_t \dot{H}^1_x}^2 \|\phi\|_{H^{s,0}} \\
\|A_0(\phi) - A_0(\psi)\|_{L^{\infty}_t \dot{H}_x^1} 
\lesssim (\|\phi\|_{H^{s,\frac{3}{4}+}} &+ \|\psi\|_{H^{s,\frac{3}{4}+}} 
+ \|\partial_t \phi\|_{H^{s-1,\frac{3}{4}+}} 
\nonumber
+
\|\partial_t \psi\|_{H^{s-1,\frac{3}{4}+}}) \\
\label{92} 
  \cdot (\|\phi -\psi\|&_{H^{s,\frac{3}{4}+}} + \|\partial_t \phi - \partial_t \psi\|_{H^{s-1,\frac{3}{4}+}}) \\
\label{93}
\|DA_0(\phi) - DA_0(\psi)\|_{H^{\half+,0}} & 
\lesssim (\|\phi\|_{H^{s,\frac{3}{4}+}} + \|\psi\|_{H^{s,\frac{3}{4}+}} \\
\nonumber
&\quad + \|\partial_t \phi\|_{H^{s-1,\frac{3}{4}+}} +
\|\partial_t \psi\|_{H^{s-1,\frac{3}{4}+}}) \\
&
\nonumber  \quad \cdot (\|\phi - \psi\|_{H^{s,\frac{3}{4}+}} + \|\partial_t \phi - \partial_t \psi\|_{H^{s-1,\frac{3}{4}+}}) \\
\nonumber
\|D(B_0(\phi,A) - B_0(\psi,A'))\|_{H^{-\half+,0}} &\\
\label{94}
 \lesssim \big( (\|\phi\|_{H^{s,\frac{3}{4}+}} + \|\psi\|_{H^{s,\frac{3}{4}+}} &+ \|\partial_t \phi\|_{H^{s-1,\frac{3}{4}+}} 
+
\|\partial_t \psi\|_{H^{s-1,\frac{3}{4}+}}) \\
\nonumber 
+(\|\phi\|_{H^{s,\frac{3}{4}+}}  + \|\psi\|_{H^{s,\frac{3}{4}+}} &+ \|A\|_{H^{s,\half+}} +
\|A'\|_{H^{s,\half+}})^2 \big) \\
\label{94}
\nonumber  \cdot (\|\phi - \psi\|_{H^{s,\frac{3}{4}+}} + \|\partial_t \phi & - \partial_t \psi\|_{H^{s-1,\frac{3}{4}+}} + \|A-A'\|_{H^{s,\half+}}) \, .
\end{align}
(\ref{92}) and (\ref{94}) were proven in Lemma \ref{Lemma9} and Lemma \ref{Lemma11}, respectively. (\ref{93}) follows by Lemma \ref{Lemma9} and Lemma \ref{Lemma10}. (\ref{86}) and (\ref{87}) were proven in \cite{T}, Prop. 9.2. \\
{\bf Proof of (\ref{88}):} By two applications of Theorem \ref{Theorem3.2} we obtain (for $s\le1$)
$$ \|uvw\|_{H^{s-1,-\frac{1}{4}+2\epsilon}} \lesssim \|uv\|_{H^{0+,\frac{1}{4}+3\epsilon}} \|w\|_{H^{s,\frac{3}{4}+}} \lesssim \|u\|_{H^{s,\frac{3}{4}+}} \|v\|_{H^{s,\frac{3}{4}+}} \|w\|_{H^{s,\frac{3}{4}+}} $$
with parameters $s_0=1-s$ , $b_0= \frac{1}{4}-2\epsilon$ , $s_1=0+$ , $b_1= \frac{1}{4}+3\epsilon$ , $s_2=s$ , $b_2 = \frac{3}{4}+$ for the first estimate, so that $s_0+s_1+s_2 = 1+$ and $s_0+s_1+s_2 + s_1+s_2 + b_0 > \frac{3}{2}$, and with $s_0=0-$ , $s_1=s_2=s$ , $ b_0 = -\frac{1}{4}-3\epsilon$,  $b_1=b_2= \frac{3}{4}+$ for the second estimate, so that $s_0+s_1+s_2 > \frac{3}{2}$ and  $s_0+s_1+s_2 + s_1+s_2 + b_0 > \frac{3}{2}$ . \\
{\bf Proof of (\ref{89}):} We first prove
\begin{equation}
\label{95}
\|uv\|_{H^{s-1}_x} \lesssim \|Du\|_{H^{\half+}_x} \|v\|_{H^{s-1}_x} \, ,
\end{equation}
which by duality is equivalent to
$$ \|uw\|_{H^{1-s}_x} \lesssim \|Du\|_{H^{\half+}_x} \|w\|_{H^{1-s}_x} \, . $$
Using the fractional Leibniz rule we have to consider
$$ \|u \Lambda^{1-s} w \|_{L^2_x} \le \|u\|_{L^{\infty}_x} \|\Lambda^{1-s} w\|_{L^2_x} \lesssim \|Du\|_{H^{\half+}_x} \|w\|_{H^{1-s}_x} $$ 
and  
\begin{align*}
\|\Lambda^{1-s} u w \|_{L^2_x} \le \|\Lambda^{1-s} u\|_{L^p_x} \|w\|_{L^q_x} & \lesssim \|\Lambda^{1-s} u\|_{\dot{H}^{s+\half}_x} \|w\|_{H^{1-s}_x} \\
& \lesssim \|Du\|_{H^{\half}_x} \|w\|_{H^{1-s}_x} \, ,
\end{align*}
where $\frac{1}{q} = \half - \frac{1-s}{3}$ , $\frac{1}{p} = \frac{1-s}{3}$ , so that $H^{1-s}_x \subset L^q_x$ and $\dot{H}^{s+\half}_x \subset L^p_x$ . This implies (\ref{95}). Thus we obtain
$$ \|A_0 \partial_t \phi\|_{L^2_t H^{s-1}_x} \lesssim \|DA_0\|_{L^2_t H^{\half+}_x} \|\partial_t \phi \|_{L^{\infty}_t H^{s-1}_x} \lesssim \|DA_0\|_{H^{\half+,0}} \|\partial_t \phi\|_{H^{s-1,\half+}} \, . $$
{\bf Proof of (\ref{90}):} We prove
\begin{equation}
\label{96}
\|uv\|_{H^{s-1}_x} \lesssim \|Du\|_{H^{-\half+}_x} \|v\|_{H^s_x} \, ,
\end{equation}
which is equivalent to
\begin{equation}
\label{97}
\|uw\|_{H^{-s}_x} \lesssim \|Du\|_{H^{-\half +}_x} \|w\|_{H^{1-s}_x} \, .
\end{equation}
For large frequencies of $u$ this follows from Prop. \ref{Prop.4.4}. For small frequencies of $u$ Prop. \ref{Prop.4.4} gives
$$ \|uw\|_{\dot{H^{-s+\half}_x}} \lesssim \|u\|_{\dot{H}^1_x} \|w\|_{\dot{H}^{1-s}_x} \, , $$
which is stronger than (\ref{97}), thus (\ref{96}) is proven. Consequently we obtain
$$ \|B_0 \phi\|_{L^2_t H^{s-1}_x} \lesssim \|DB_0\|_{L^2_t H^{-\half+}_x} \|\phi\|_{L^{\infty}_t H^s_x} \lesssim \|DB_0\|_{H^{-\half+,0}} \|\phi\|_{H^{s,\half+}} \, . $$
{\bf Proof of (\ref{91}):} By Prop. \ref{Prop.4.4} we obtain
$$ \|A_0A_0 \phi\|_{L^2_t H^{s-1}_x} \lesssim \|A_0 A_0\|_{L^{\infty}_t \dot{H}^{\half}_x} \|\phi\|_{L^2_t H^s_x} \lesssim \|A_0\|^2_{L^{\infty}_t \dot{H}^1_x} \|\phi\|_{H^{s,0}} \, . $$
\end{proof}

\section{Unconditional uniqueness in Coulomb gauge}
Our main theorem reads as follows.

\begin{theorem}
\label{Theorem3}
Assume $ s > \frac{25 - \sqrt{313}}{8} \approx 0.91254 $ . Let data $a_{0j} , \phi_0 \in H^s$ , $a_{1j} , \phi_1 \in H^{s-1}$ be given ($j=1,2,3$), which fulfill the compatability condition (\ref{CC}). Then there exists $T >0$ , such that (\ref{b}),(\ref{c}),(\ref{d'}),(\ref{e}) with initial condition (\ref{IC}) has an (unconditionally) unique solution
$$ \phi, A_j \in C^0([0,T],H^s) \cap C^1([0,T],H^{s-1}) \, , \, A_0 \in C^0([0,T],\dot{H}^1) \cap C^1([0,T],L^2) \, .$$
\end{theorem}
\begin{proof} In view of Theorem \ref{Theorem2} we only have to prove that any solution in this regularity class satisfies
$ \phi,A_j \in H^{\frac{3}{4}+,\frac{3}{4}+}[0,T] \, , \, \partial_t \phi, \partial_t A_j \in H^{-\frac{1}{4}+,\frac{3}{4}+}[0,T] $ . We may assume $s,T \le 1$ . We frequently omit $[0,T]$ in terms of the form $H^{s,b}[0,T]$.\\
{\bf 1.} In a first step we prove $A_j \in H^{\frac{5}{3}s-1,\frac{s}{3}+\half-}[0,T]$ , $\partial_t A_j \in  H^{\frac{5}{3}s-2,\frac{s}{3}+\half-}[0,T]$ by use of (\ref{b}). \\
Using (\ref{e'}) the term $\partial_j \partial_t A_0$ in all the $H^{s,b}$-norms used in the sequel (namely its Fourier transform) can be handled like the terms $\phi \overline{\partial_j \phi}$ and $|\phi|^2 A_j$, so that we only have to consider these terms here and in the following steps. \\ 
{\bf a.} We have
$$ \|\phi \overline{\partial_j \phi} \|_{H^{\frac{5}{3}s-2,\frac{s}{3}-\half-}} \lesssim \|\phi \overline{\partial_j \phi} \|_{L^2_t H_x^{s-1,\frac{3}{3-s}}} $$
by Prop. \ref{Lemma7.1} with $\frac{1}{r} = \frac{s}{3}$ , $q=2$ , so that $\frac{2}{r}-1 = \frac{2}{3}s-1$ , $\frac{1}{r'} = 1-\frac{s}{3}=\frac{3-s}{3}$ . We want to prove
$$ \| \phi \overline{\partial_j \phi} \|_{H^{s-1,\frac{3}{3-s}}_x} \lesssim \|\phi\|_{H^s_x} \|\partial_j \phi\|_{H^{s-1}_x} \, , $$
which is equivalent to
$$ \|uv\|_{H^{1-s}_x} \lesssim \|u\|_{H^s_x} \|v\|_{H^{1-s,\frac{3}{s}}_x} \, . $$
Using the fractional Leibniz rule we have to estimate two terms:
$$ \| u \Lambda^{1-s} v \|_{L^2_x} \lesssim \|u\|_{L^{\frac{6}{3-2s}}_x} \|\Lambda^{1-s} v\|_{L^{\frac{3}{s}}_x} \lesssim \|u\|_{H^s_x} \|v\|_{H^{1-s,\frac{3}{s}}_x} $$
by the Sobolev embedding $H^s_x \hookrightarrow L^{\frac{6}{3-2s}}_x$ , and
 $$\| \Lambda^{1-s} u v \|_{L^2_x} \lesssim \| \Lambda^{1-s} u \|_{L^p_x} \|v\|_{L^q_x} \lesssim \|\Lambda^{1-s}u\|_{H^{2s-1}_x} \|v\|_{H^{1-s,\frac{3}{s}}_x} \lesssim \|u\|_{H^s_x}  \|v\|_{H^{1-s,\frac{3}{s}}_x} \, , $$
where $\frac{1}{p} = \half - \frac{2s-1}{3}$ , $\frac{1}{q} = \frac{s}{3}-\frac{1-s}{3}$ , so that $H^{2s-1}_x \hookrightarrow L^p_x$ and $H^{1-s,\frac{3}{s}}_x \hookrightarrow L^q$ .
We conclude
$$\|\phi \overline{\partial_j \phi} \|_{H^{\frac{5}{3}s-2,\frac{s}{3}-\half-}} \lesssim \|\phi\|_{L^{\infty}_t H^s_x} \|\partial_j \phi\|_{L^{\infty}_t H^{s-1}_x} < \infty \, . $$
{\bf b.} Next we obtain similarly
\begin{align*}
\| |\phi|^2 A_j\|_{H^{\frac{5}{3}s-2,\frac{s}{3}-\half-}} & \lesssim \| |\phi|^2 A_j\|_{L^2_t H^{s-1,\frac{3}{3-s}}_x} \lesssim \| |\phi|^2 A_j\|_{L^2_t L^{\frac{3}{4-2s}}_x} \\
& \lesssim \|\phi\|^2_{L^{\infty}_t L^{\frac{9}{4-2s}}_x} \|A_j\|_{L^{\infty}_t H^{\frac{9}{4-2s}}_x} \lesssim \|\phi\|^2_{L^{\infty}_t H^s_x} \|A_j\|_{L^{\infty}_t H^s_x} < \infty 
\end{align*}
by the embeddings $H^{1-s,\frac{3}{4-2s}}_x \hookrightarrow L^{\frac{3}{3-s}}_x$ and $H^s_x \hookrightarrow L^{\frac{9}{4-2s}}_x$ , which hold for $s \ge \half$ . \\
{\bf a.} and {\bf b.} imply $A^{inh}_j \in H^{\frac{5}{3}s-1,\frac{s}{2}+\half-}$, and moreover we know  $A^{hom}_j \in H^{s,1-}$ , so that $A_j \in H^{\frac{5}{3}s-1,\frac{s}{3}+\half-}$     Here $A^{inh}_j$ and $A^{hom}_j$ denotes the inhomogeneous and homogenous part of $A_j$, respectively. \\
{\bf 2.} Next we prove $\phi \in H^{\frac{5}{3}s-1,\frac{s}{3}+\half-}$ , $\partial_t \phi \in H^{\frac{5}{3}s-2,\frac{s}{3}+\half-}$by use of (\ref{c}). \\
{\bf a.} The term $A^j \partial_j \phi$ is handled like $\phi \overline{\partial_j \phi}$ before. \\
{\bf b.} The term $(\partial_t A_0)\phi $ is also treated like $\phi \overline{\partial_j \phi}$ using $\partial_t A_0 \in C^0_t L^2_x$ , which is even more regular than $\partial_j \phi \in C^0_t H^{s-1}_x$ . \\
{\bf c.} We crudely estimate by Prop. \ref{Lemma7.1}
$$ \|A_0 \partial_t \phi\|_{H^{\frac{5}{3}s-2,\frac{s}{3}-\half-}} \lesssim \|A_0 \partial_t \phi\|_{H^{s-\frac{4}{3},-\frac{1}{6}-}} \lesssim \|A_0 \partial_t \phi\|_{L^2_t H^{s-1,\frac{3}{2}}_x} $$
with parameters $r=3$ , $q=2$ , so that $\frac{2}{r}-1= - \frac{1}{3}$ , using $\frac{5}{3}s-2 \le s- \frac{4}{3}$ and $\frac{s}{3}-\half \le - \frac{1}{6}$ for $ s \le 1$ . Now we establish
\begin{equation}
\label{10}
\|A_0 \partial_t \phi\|_{H^{s-1,\frac{3}{2}}_x} \lesssim \|DA_0\|_{L^2_x} \|\partial_t \phi\|_{H^{s-1}_x} \, ,
\end{equation}
which is equivalent to
$$ \|uv\|_{H^{1-s}_x} \lesssim \|Du\|_{L^2_x} \|v\|_{H^{1-s,3}_x} \, . $$
We obtain
$$ \|u \Lambda^{1-s} v\|_{L^2_x} \lesssim \|u\|_{L^6_x} \|\Lambda^{1-s} v\|_{L^3_x} \lesssim \|Du\|_{L^2_x} \|v\|_{H^{1-s,3}_x} $$
and 
\begin{align*}
\| \Lambda^{1-s}u v\|_{L^2_x} \lesssim \|uv\|_{L^2_x} + \|D^{1-s}u v\|_{L^2_x} &\lesssim \|u\|_{L^6_x} \|v\|_{L^3_x} + \|D^{1-s}u\|_{L^p_x} \|v\|_{L^q_x} \\
& \lesssim \|Du\|_{L^2_x}\|v\|_{H^{1-s,3}_x} \quad , 
\end{align*}
where $\frac{1}{p}= \half - \frac{s}{3}$ , $\frac{1}{q}=\frac{s}{3}$ , so that $\dot{H}^{1,2}_x \hookrightarrow \dot{H}^{1-s,p}_x$ and $H^{1-s,3}_x \hookrightarrow L^q$ . Thus we obtain
$$ \|A_0 \partial_t \phi\|_{H^{\frac{5}{3}s-2,\frac{s}{3}-\half-}} \lesssim \|DA_0\|_{L^{\infty}_t L^2_x} \|\partial_t \phi\|_{L^{\infty}_t H^{s-1}_x} < \infty \,. $$
{\bf d.} We obtain as in 1b. 
\begin{align*}
\|A_0 A_0 \phi\|_{H^{\frac{5}{3}s-2,\frac{s}{3}-\half-}} & \lesssim \|A_0 A_0 \phi\|_{L^2_t L^{\frac{3}{4-2s}}_x} \lesssim \|A_0\|_{L^{\infty}_t L^6_x}^2 \|\phi\|_{L^{\infty}_t L^{\frac{3}{3-2s}}_x} \\
& \lesssim \|DA_0\|_{L^{\infty}_t L^2_x}^2 \|\phi\|_{L^{\infty}_t H^s_x} < \infty \, ,
\end{align*}
because $H^s_x \hookrightarrow L^{\frac{3}{3-2s}}_x$ , if $s \le \frac{3}{2}$ .\\
{\bf e.} The term $A_j A_j \phi$ can be handled like $|\phi|^2 A_j$ in 1b. \\
{\bf a.} - {\bf e.} imply the claimed regularity of $\phi$ and $\partial_t \phi$ . \\[0.5em]
{\bf 3.} Now we interpolate the properties of $A_j$ and $\phi$. We obtain
$$ A_j,\phi \in H^{s,0} \cap H^{\frac{5}{3}s-1,\frac{s}{3}+\half-} \subset H^{2-s-\frac{2}{r_1},\half-\frac{1}{r_1}+} \, , $$
where we choose the interpolation parameter $\theta$ such that
$$\theta(\half+\frac{s}{3}) = \half - \frac{1}{r_1}+ \Longleftrightarrow \theta = \frac{\half-\frac{1}{r_1}}{\half+\frac{s}{3}}+$$
 and 
$$ (1-\theta)s + \theta(\frac{5}{3}s-1) = s-\theta(1-\frac{2}{3}s) = 2-s-\frac{2}{r_1} \, . $$
This defines $r_1$ as
$$\frac{2}{r_1} = 2(1-s) + \frac{\half-\frac{1}{r_1}}{\half+\frac{s}{3}} (1-\frac{2}{3}s)- \, . $$
An elementary calculation gives
\begin{equation}
\label{r1}
\frac{1}{r_1} = \frac{3}{4}-\frac{1}{3}s-\frac{1}{3}s^2 + \, . 
\end{equation}
This is the first step of an iteration which gives more and more regularity of $A_j$ and $\phi$. The general iteration step is contained in the next Lemma.\\

We define a sequence $\{r_k\}$ iteratively by (\ref{r1}) and for $k \in{\mathbb N}$ by
\begin{equation}
\label{r(k+1)}
\frac{2}{r_{k+1}} = 2-2s + \frac{\half - \frac{1}{r_{k+1}}}{1-\frac{1}{r_k}+\frac{1-s}{3}} \Big(\frac{2}{r_k} - \frac{2(1-s)}{3} \Big)+ \, . 
\end{equation}

\begin{lemma}
\label{LemmaL1}
Assume $$A_j, \phi \in H^{2-s-\frac{2}{r_k},\half - \frac{1}{r_k}+} \, .$$  
Then we also have
$$ A_j,\phi \in H^{s+1-\frac{2}{q_k},\frac{3}{2}-\frac{1}{q_k}-} \, , $$ 
where $r_k$ is defined by (\ref{r1}) and (\ref{r(k+1)}),  and $\frac{1}{q_k} = \frac{1}{r_k} + \half - \frac{1-s}{3} \, . $
\end{lemma}
\noindent{\bf Remark.} By  Remark 5.1 we have $ \half > \frac{1}{r_k} >1-s \, .$
\begin{proof}[Proof of Lemma \ref{LemmaL1}:] 
Fundamental for the proof is the fact that our assumption implies by Prop. \ref{Lemma7.1} :
$ A_j, \phi \in L^2_t H^{1-s,r_k}_x \, .$ \\
{\bf 1.} The claimed regularity of $A_j$ reduces by Prop. \ref{Lemma7.1} to proving that the right hand side of (\ref{b}) belongs to $L^2_t H^{s-1,q_k}_x$ , because  $ L^2_t H^{s-1,q_k}_x  \hookrightarrow H^{s-\frac{2}{q_k},\half-\frac{1}{q_k}-}$ . \\
{\bf a.} We want to prove
$$ \| \phi \partial_j \phi\|_{L^2_t H^{s-1,q_k}_x} \lesssim \|\phi\|_{L^2_t H^{1-s,r_k}_x} \|\partial_j \phi\|_{L^{\infty}_t H^{s-1}_x} < \infty \, , $$
which follows from
$$ \|uv\|_{H^{1-s}_x} \lesssim \|u\|_{H^{1-s,q_k'}_x} \|v\|_{H^{1-s,r_k}_x} \, . $$
Using the fractional Leibniz rule we estimate
$$ \|\Lambda^{1-s}u v\|_{L^2_x} \lesssim \|\Lambda^{1-s}u\|_{L^{q_k'}_x} \|v\|_{L^p_x} \lesssim \|u\|_{H^{1-s,q_k'}_x} \|v\|_{H^{1-s,r_k}_x} \, , $$
where
$\frac{1}{p} = \frac{1}{r_k} - \frac{1-s}{3} \, , $ so that $H^{1-s,r_k}_x \hookrightarrow L^p_x $ , and $\frac{1}{q_k'} = \half - \frac{1}{p} \, , $ so that $\frac{1}{q_k} = \frac{1}{r_k} + \half - \frac{1-s}{3}$ . Moreover
$$ \|u \Lambda^{1-s}v\|_{L^2_x} \lesssim \|u\|_{L^q_x} \|\Lambda^{1-s} v\|_{L^{r_k}_x} \lesssim \|u\|_{H^{1-s,q_k'}_x} \|v\|_{H^{1-s,r_k}_x} \, , $$
where $\frac {1}{q} = \half - \frac{1}{r_k}$ , so that $H^{1-s,q_k'}_x \hookrightarrow L^q_x$ , because $\half -\frac{1}{r_k} = \frac{1}{q_k'} - \frac{1-s}{3}$ . \\
{\bf b.} Next we obtain as in a. and by Prop. \ref{Prop.4.4} using $s >\half$ :
\begin{align*}
\| |\phi|^2 A_j\|_{L^2_t H^{s-1,q_k}_x} & \lesssim \| |\phi|^2 \|_{L^{\infty}_t H^{s-1}_x} \|A_j\|_{L^2_t H^{1-s,r_k}_x} \\
& \lesssim \|\phi\|^2_{L^{\infty}_t H^s_x} \|A_j\|_{L^2_t H^{1-s,r_k}_x} < \infty
\end{align*} 
{\bf 2.} We now prove that $ \phi \in H^{s+1-\frac{2}{q_k},\frac{3}{2}-\frac{1}{q_k}-} $ . As in 1. this reduces to proving that the right hand side of (\ref{c}) belongs to $L^2_t H^{s-1,q_k}_x$ .\\
{\bf a.} The term $A_j \partial_j \phi$ can be handled like $\phi \partial_j \phi$ in 1a. \\
{\bf b.} Before considering the term $A_0 \partial_t \phi$ we have to prove some more regularity for $A_0$ ,  considering low and high frequencies separately, denoted by $A^l_0$ and $A^h_0$ . Using the assumption $A_0 \in C^0_t \dot{H}^1_x \subset C^0_t L^6_x$ we immediately obtain $A^l_0 \in C^0_t \dot{H}^m_x \subset C^0_t L^p_x$ $\forall \, m \ge 1 \, \forall \,\, 6 \le p \le \infty$ .

For $A^h_0$ we use (\ref{d}) and obtain
$$
\|A^h_0\|_{H^{2s-\half-}_x} \lesssim \|\Delta A_0\|_{H^{2s-\frac{5}{2}-}_x} \lesssim \|\phi \partial_t \phi \|_{H^{2s-\frac{5}{2}-}_x} + \|A_0 |\phi|^2|_{H^{2s-\frac{5}{2}-}_x} \, . $$
Now by Prop. \ref{Prop.4.4} 
\begin{align*}
 \|\phi \partial_t \phi \|_{H^{2s-\frac{5}{2}-}_x} & \lesssim \|\phi\|_{H^s_x} \|\partial_t \phi\|_{H^{s-1}_x} \\
 \|A_0 |\phi|^2|_{H^{2s-\frac{5}{2}}_x} & \lesssim \|A_0 |\phi|^2\|_{\dot{H}^{2s-\frac{5}{2}}_x} \lesssim \|A_0\|_{\dot{H}^1_x} \| |\phi|^2 \|_{\dot{H}^{2s-2}_x} \\
& \lesssim \|A_0\|_{\dot{H}^1_x} \|\phi\|^2_{\dot{H}^{s-\frac{1}{4}}_x} \lesssim \|DA_0\|_{L^2_x} \|\phi\|^2_{H^s_x} \, ,
\end{align*}
so that $A^h_0 \in L^{\infty}_t H^{2s+\half-}_x$ . 

We want to establish
$$ \|A^l_0 \partial_t \phi\|_{L^2_t H^{s-1,q_k}_x} \lesssim \|A^l_0\|_{L^{\infty}_t L^{p_k}_x} \|\partial_t \phi\|_{L^{\infty}_t H^{s-1}_x} < \infty$$
for some $ 6 \le p_k < \infty $ . By duality we have to prove
$$ \|A^l_0 v\|_{H^{1-s}_x} \lesssim \|A_0^l\|_{L^{p_k}_x} \|v\|_{H^{1-s,q_k'}_x} \, . $$
Now we obtain
\begin{align*}
\|\Lambda^{1-s} A^l_0 v \|_{L^2_x} & \lesssim \|\Lambda^{1-s} A^l_0\|_{L^{p_k}_x} \|v\|_{L^{q_k'}_x} \lesssim \|A^l_0\|_{L^{p_k}_x} \|v\|_{H^{1-s,q_k'}_x} \\
\|A^l_0 \Lambda^{1-s} v\|_{L^2_x} & \lesssim \|A^l_0\|_{L^{p_k}_x} \|\Lambda^{1-s} v \|_{L^{q_k'}_x} \, ,
\end{align*}
where 
$
\frac{1}{p_k} = \half-\frac{1}{q_k'} = \frac{1}{q_k}-\half = \frac{1}{r_k} - \frac{1-s}{3} \, .
$
By Lemma \ref{LemmaA.2} we have $r_{k+1} \ge r_k$,  so that $\frac{1}{p_k} \le \frac{1}{r_1}-\frac{1-s}{3} = \frac{3}{4} - \frac{1}{3}s - \frac{1}{3}s^2 - \frac{1}{3} + \frac{1}{3}s - = \frac{5}{12} - \frac{1}{3}s^2 - \le \frac{1}{6} $ , if $s^2 \ge \frac{3}{4} \Leftrightarrow s \ge \half \sqrt{3} \approx 0.866 $ , which is fulfilled by assumption.

Next we prove as in 1a:
\begin{align*}
\| A^h_0 \partial_t \phi\|_{L^2_t H^{s-1,q_k}_x} &\lesssim \|A^h_0\|_{L^2_t H^{1-s,r_k}_x} \|\partial_t \phi\|_{L^{\infty}_t H^{s-1}_x} 
\lesssim \|A^h_0\|_{L^{\infty}_t H^{2s-\half-}_x} \|\partial_t \phi\|_{L^{\infty}_t H^{s-1}_x} \\
& < \infty \, , 
\end{align*}
because $H^{2s-\half-}_x \hookrightarrow H^{1-s,r_k}_x$ , if $\frac{1}{r_k} > 1-s$ . This is the case, because by Lemma \ref{LemmaA.1} and (\ref{rinf}) we know $\frac{1}{r_k} \ge \frac{1}{r_{\infty}} = \frac{(\frac{7}{3}-\frac {2}{3}s)(1-s)}{3-2s}+ > 1-s$ for $s > \half$ . \\
{\bf c.} The term $\partial_t A_0 \phi$ is handled like the term $\phi \partial_t \phi$ in 1a, using that $\partial_t A_0 \in C^0_t L^2_x \subset C^0_t H^{s-1}_x$ . \\
{\bf d.} In order to estimate $A_jA_j \phi$ we start as in 1a. und use Prop. \ref{Prop.4.4}:
\begin{align*}
\|A_j A_j \phi\|_{L^2_t H^{s-1,q_k}_x} & \lesssim \|A_j\|_{L^2_t H^{1-s,r_k}_x} \|A_j \phi\|_{L^{\infty}_t H^{s-1}_x} \\
& \lesssim  \|A_j\|_{L^2_t H^{1-s,r_k}_x} \|A_j\|_{L^{\infty}_t H^s_x} \|\phi\|_{L^{\infty}_t H^s_x} < \infty\, .
\end{align*}
{\bf e.} Finally we have to treat $A_0 A_0 \phi$ considering low and high frequencies of $A_0$
. \\
1. We estimate 
\begin{align*}
\|A_0^h A_0^h \phi\|_{L^2_t H^{s-1,q_k}_x} & \lesssim \|\phi\|_{L^2_t H^{1-s,r_k}_x} \|A_0^h A_0^h\|_{L^{\infty}_t H^{s-1}_x} \\
& \lesssim \|\phi\|_{L^2_t H^{1-s,r_k}_x} \|A_0^h \|_{L^{\infty}_t H^{s}_x}^2 < \infty \, ,
\end{align*} 
where the first inequality is obtained as in 1a. and the second inequality by Prop. \ref{Prop.4.4} using $A_0^h \in L^{\infty}_t H^s_x$ . \\
2. Similarly
\begin{align*}
\|A_0^h A_0^l \phi\|_{L^2_t H^{s-1,q_k}_x} & \lesssim \|\phi\|_{L^2_t H^{1-s,r_k}_x} \|A_0^h A_0^l\|_{L^{\infty}_t H^{s-1}_x} \\
& \lesssim \|\phi\|_{L^2_t H^{1-s,r_k}_x} \|A_0^h \|_{L^{\infty}_t L^2_x} \|A_0^l\|_{L^{\infty}_t L^{\infty}_x} < \infty \, .
\end{align*}
3. Finally
$$ \|A_0^l A_0^l \phi\|_{L^2_t H^{s-1,q_k}_x} \lesssim \|A_0^l A_0^l \phi\|_{L^2_t L^{q_k}_x} \lesssim \|A_0^l\|_{L^{\infty}_t L^{t_k}_x}^2 \|\phi\|_{L^2_t L^2_x} < \infty \, , $$
where $\frac{2}{t_k} = \frac{1}{q_k} - \half = \frac{1}{r_k} - \frac{1-s}{3} \le \frac{1}{r_1} - \frac{1-s}{3} \le \frac{1}{6}$ as in 2b, so that $t_k \ge 12 $, and thus as remarked before, $A_0^l \in L^{\infty}_t L^{t_k}_x$ . This completes the proof of Lemma \ref{LemmaL1}.
\end{proof}

\begin{lemma}
\label{LemmaL2}
Under the assumptions of Lemma \ref{LemmaL1} we obtain the regularity
$A_j,\phi \in H^{2-s-\frac{2}{r_{k+1}},\half - \frac{1}{r_{k+1}}+}$ , where $r_{k+1}$ is defined by (\ref{r1}) and (\ref{r(k+1)}).
\end{lemma}
\begin{proof}
We interpolate the property $A_j,\phi \in H^{s+1-\frac{2}{q_k},\frac{3}{2}-\frac{1}{q_k}-} \, , $ where
\begin{equation}
\label{qk}
 \frac{1}{q_k} = \frac{1}{r_k} + \half - \frac{1-s}{3} \, , 
\end{equation}
 with our assumption $A_j,\phi \in H^{s,0}$ and obtain  the claimed regularity, if we choose the interpolation parameter $\theta$ by the conditions $\theta(\frac{3}{2}-\frac{1}{q_k}-) = \half - \frac{1}{r_{k+1}}+ $ and $(1-\theta)s + \theta(s+1-\frac{2}{q_k}) = s - \theta(\frac{2}{q_k}-1) = 2-s - \frac{2}{r_{k+1}} $ . This requires
\begin{equation}
\label{r(k+1)'}
 \frac{2}{r_{k+1}} = 2-2s + \frac{\half - \frac{1}{r_{k+1}}}{\frac{3}{2}-\frac{1}{q_k}} \Big(\frac{2}{q_k} -1 \Big) +\, , 
\end{equation}
which is equivalent to (\ref{r(k+1)}).
\end{proof}
 We may apply Lemma \ref{LemmaL1} and Lemma \ref{LemmaL2} iteratively to conclude
$$ A_j,\phi \in H^{s+1-\frac{2}{q_k},\frac{3}{2}-\frac{1}{q_k}-} \quad \forall \,k\in \mathbb N \, ,$$
 where $\frac{1}{q_k} = \frac{1}{r_k} + \half - \frac{1-s}{3} \, . $
\\[0.2em]

In the appendix 5.1 we establish $q_k \to q_{\infty}$ and $r_k \to r_{\infty}$ and determine these limits explicitly. Assuming this we immediately conclude
$$ A_j,\phi \in H^{s+1-\frac{2}{q_\infty}-,\frac{3}{2}-\frac{1}{q_{\infty}}-} \, . $$
As a consequence we obtain

\begin{lemma}
\label{LemmaL4} 
$A_j,\phi \in H^{s_{\infty}-,\frac{3}{4}+} \, , $
where $s_{\infty} = s - \frac{\frac{3}{4}(\frac{2}{q_{\infty}}-1)}{\frac{3}{2}-\frac{1}{q_{\infty}}} -= s - \frac{6(1-s)}{3(3-2s) - 4(1-s)}-$ .
\end{lemma}
\begin{proof}
This follows by interpolation. We obtain
$$ A_j,\phi \in H^{s+1-\frac{2}{q_{\infty}}-,\frac{3}{2}-\frac{1}{q_{\infty}}-} \cap H^{s,0} \subset H^{s_{\infty}-,\frac{3}{4}+} \, , $$
if we choose the interpolation parameter  $\theta$ such that $\theta(\frac{3}{2} - \frac{1}{q_{\infty}}-) = \frac{3}{4}+$ , so that 
\begin{equation}
\label{sinf}
s_{\infty} = (1-\theta)s + \theta(s+1-\frac{2}{q_{\infty}}) = s-\theta(\frac{2}{q_{\infty}}-1) =  s - \frac{\frac{3}{4}(\frac{2}{q_{\infty}}-1)}{\frac{3}{2}-\frac{1}{q_{\infty}}} -\, . 
\end{equation}
Using $q_k \to q_{\infty}$ , $r_k \to r_{\infty}$ , we easily obtain from (\ref{r(k+1)}):
$$ \frac{2}{r_{\infty}} = 2-2s+ \frac{\half - \frac{1}{r_{\infty}}}{1-\frac{1}{r_{\infty}}+\frac{1-s}{3}} \Big(\frac{2}{r_{\infty}} - \frac{2(1-s)}{3} \Big)+ \, .$$
By an elementary calculation we obtain
\begin{equation}
\label{rinf}
\frac{1}{r_{\infty}} = \frac{(\frac{7}{3}-\frac{2}{3}s)(1-s)}{3-2s}+ \, , 
\end{equation}
and by the definition of $q_k$  in Lemma \ref{LemmaL1} we have 
$$ \frac{1}{q_{\infty}} = \frac{1}{r_{\infty}} + \half - \frac{1-s}{3} +\, . $$
Inserting this in (\ref{sinf})  an easy calculation yields:
$$ s_{\infty} = s - \frac{6(1-s)}{3(3-2s) - 4(1-s)}- \, .$$
\end{proof}

One easily checks that $s_{\infty} > \frac{3}{4}$ , if $s > \frac{1}{8} (25 - \sqrt{313}) \approx 0.91352$ , so that by Lemma \ref{LemmaL4} the solution belongs to a regularity class where uniqueness is known by Theorem \ref{Theorem1}. Thus the unconditional uniqueness result of Theorem \ref{Theorem3} is now completely proven. 
\end{proof}

\section{Unconditional uniqueness in Lorenz gauge}
Consider the Maxwell-Klein-Gordon system (\ref{1}),(\ref{2}) together with the Lorenz gauge 
\begin{equation}
\label{LG}
\partial^{\mu} A_{\mu} = 0 \, .
\end{equation}
It is easy to see that in this case the system (\ref{1}),(\ref{2}) can be reformulated as (cf. \cite{ST}):
\begin{align}
\label{b1}
\square A_{\mu} & = - Im(\phi \overline{\partial_{\mu} \phi}) - |\phi|^2 A_{\mu} \\
\label{c1}
(\square - m^2) \phi & = 2i A^{\mu} \partial_{\mu} \phi + A_{\mu} A^{\mu} \phi \, .
\end{align}
We want to solve this system together with the initial conditions
\begin{equation}
\label{IC0}
A_{\mu}(0) = a_{0 \mu} \, , \, (\partial_t A_{\mu})(0) = \dot{a}_{0 \mu} \, , \, \phi(0) = \phi_0 \, , \, (\partial_t \phi)(0) = \phi_1 \, ,
\end{equation}
where
\begin{equation}
\label{IC1}
a_{00} = \dot{a}_{00} = 0
\end{equation}
and the following compatibility condition holds:
\begin{equation}
\label{IC2}
\partial^k a_{0k} = 0 \, .
\end{equation}
(\ref{IC1}) may be assumed, because otherwise the potential is not uniquely determined by the Lorenz condition.

Define $\phi_{\pm} = \half(\phi \pm (i \Lambda_m)^{-1} \partial_t \phi)$ , so that $\phi = \phi_+ + \phi_-$ , $\partial_t \phi = i \Lambda_m(\phi_+ - \phi_-) $,  and $A_{\mu \pm} = \half(A_{\mu} \pm (i \Lambda)^{-1} \partial_t A_{\mu})$ , so that $A_{\mu} = A_{\mu +} + A_{\mu -}$ , $\partial_t A_{\mu} = i \Lambda (A_{\mu +} - A_{\mu -}) $. We reformulate (\ref{b1}),(\ref{c1}) as the following first order (in $t$) system:
\begin{align}
\label{b2}
(i\partial_t \pm \Lambda) A_{\mu \pm} & = - (\pm 2 \Lambda)^{-1} \mathcal{N}_{\mu} (\phi_+,\phi_-,A_+,A_-) \\
\label{c2}
(i\partial_t \pm \Lambda_m) \phi_{\pm} & = - (\pm 2 \Lambda_m)^{-1} \mathcal{M} (\phi_+,\phi_-,A_+,A_-)\, ,
\end{align}
where
\begin{align*}
{\mathcal N}_0 (\phi_+,\phi_-,A_+,A_-) & = -Im(\phi \, i \Lambda_m(\overline{\phi}_+ - \overline{\phi}_-)) -A_0 |\phi|^2 + A_0 \\
{\mathcal N}_j (\phi_+,\phi_-,A_+,A_-) & = -Im(\phi \overline{\partial_j \phi}) - A_j |\phi|^2 + A_j \\
{\mathcal M} (\phi_+,\phi_-,A_+,A_-) & = 2i A^{\mu} \partial_{\mu} \phi + A_{\mu} A^{\mu} \phi \, .
\end{align*}
The initial conditions are
\begin{align*}
\phi_{\pm}(0) & = \half(\phi_0 \pm (i \Lambda_m)^{-1} \phi_1 )\\
A_{0 \pm}(0)  & = \half(a_{00} \pm (i\Lambda)^{-1} \dot{a}_{00}) = 0 \\
A_{j \pm}(0) & = \half(a_{0j} \pm (i\Lambda)^{-1} \dot{a}_{0j}) \, .
\end{align*}
We now assume the following regularity for the data:
$$ \phi_0 \in H^s \, , \, \phi_1 \in H^{s-1} \, , \, Da_{0j} \in H^{2s-2-} \, , \, \dot{a}_{0j} \in H^{2s-2-} \, . 
$$This implies $\phi_{\pm}(0) \in H^s$ and $DA_{j\pm}(0) \in H^{2s-2-}$ .

The following local well-posedness result was proven in \cite{P1}, Theorem 3.1.

\begin{theorem}
\label{Theorem4}
Assume $ 2 \ge s > \frac{3}{4}$ . Let $\phi_{\pm}(0) \in H^s$ , $DA_{\pm}(0) \in H^{2s-2-}$ be given. In (3+1) dimensions the system (\ref{b2}),(\ref{c2}) with initial data $\phi_{\pm}(0)$ , $A_{\pm}(0)$ has a (conditionally) unique local solution
$$ \phi_{\pm}
 \in X^{s,\half+}_{\pm}[0,T] \, , \, DA^{hom}_{\pm} \in X^{2s-2-,1-}_{\pm}[0,T] \, , \, A^{inh}_{\pm} \in X^{2s-1-,1-}_{\pm}[0,T] \, . $$
This solution fulfills
\begin{equation}
\label{RC}
\phi_{\pm} \in C^0([0,T],H^s) \, , \, DA^{hom}_{\pm} \in C^0([0,T],H^{2s-2-}) \, , \, A^{inh}_{\pm} \in C^0([0,T],H^{2s-1-}) \, .
\end{equation}
Here $A_{\pm} = (A_{0 \pm},...,A_{3 \pm})$ , $A_{\pm}^{hom}$ and $A_{\pm}^{inh}$ denotes the homogeneous and inhomogeneous part of $A_{\pm}$, respectively.
\end{theorem}
{\bf Remark:} In the case $s > 2$ a similar result holds with slightly changed regularity of the potential. We do not consider this case, because unconditional uniqueness for data with finite energy ($s \ge 1$) is well-known.\\[0.2em]

Our aim is the prove
\begin{theorem}
\label{Theorem5}
Assume $ 2 \ge s > s_0 $ , where $s_0$ is the largest root of the cubic equation $32s^3-152s^2+106s+5 = 0$  $(0.906 <s_0 < 0.907)$ . Then there exists an (unconditionally) unique solution of (\ref{b2}),(\ref{c2}) with initial data $\phi_{\pm}(0) \in H^s$ , $DA_{\pm}(0) \in H^{2s-2-}$ in the class (\ref{RC}).
\end{theorem}
\begin{proof}
We only have to demonstrate that any solution with this regularity satisfies
$$\phi_{\pm}
 \in X^{\frac{3}{4}+,\half+}_{\pm}[0,T] \, , \, DA^{hom}_{\pm} \in X^{-\half+,1-}_{\pm}[0,T] \, , \, A^{inh}_{\pm} \in X^{\half+,1-}_{\pm}[0,T] \, ,$$
because then the result is a consequence of Theorem \ref{Theorem4}. We may assume in the sequel $s \le 1$ .\\
{\bf 1.} In a first step we prove $\phi_{\pm} \in X^{\frac{7}{3}s-\frac{5}{3}-,\frac{1}{6}+\frac{2}{3}s-}$ . It suffices to prove
\begin{align*}
 &\| {\mathcal M}(\phi_+,\phi_-,A_+,A_-)\|_{H^{\frac{7}{3}s-\frac{8}{3}-,\frac{2}{3}s-\frac{5}{6}-}} \\
& \quad \quad\lesssim c(\|\phi_{\pm}\|_{L^{\infty}_t H^s_x} , \|DA^{hom}_{\pm} \|_{L^{\infty}_t H^{2s-2-}_x} , \|A_{\pm}^{inh}\|_{L^{\infty}_t H^{2s-1-}_x}) \, ,
\end{align*}
where $c$ is a continuous function. \\
{\bf a.} First we consider the term $A^{\mu} \partial_{\mu} \phi = A^{\mu}_{inh} \partial_{\mu} \phi + A^{\mu}_{hom} \partial_{\mu} \phi $ . \\
{\bf a1.} By Prop. \ref{Lemma7.1} with parameters $r' = \frac{3}{4-2s}-$ , $q=2$ we have
$$ \|A^{\mu}_{inh} \partial_{\mu} \phi\|_{H^{\frac{7}{3}s-\frac{8}{3}-,\frac{2}{3}s-\frac{5}{6}-}} \lesssim \|A^{\mu}_{inh} \partial_{\mu} \phi\|_{L^2_t H^{s-1,\frac{3}{4-2s}-}_x} \, . $$ 
We prove
$$ \|A^{\mu}_{inh} \partial_{\mu} \phi\|_{L^2_t H^{s-1,\frac{3}{4-2s}-}_x} \lesssim \|A^{\mu}_{inh}\|_{L^{\infty}_t H^{2s-1-}_x} \| \partial_{\mu} \phi\|_{L^{\infty}_t H^{s-1}_x} \, . $$
By duality it suffices to prove
$$ \| uv \|_{H^{1-s}_x} \lesssim \|u\|_{H^{2s-1-}_x} \|v\|_{H_x^{1-s,\frac{3}{2s-1}+}} \, . $$
By the fractional Leibniz rule we estimate two terms. First
$$ \|u \Lambda^{1-s}v\|_{L^2_x} \lesssim \|u\|_{L^q_x} \|\Lambda^{1-s}v\|_{L^{\frac{3}{2s-1}+}_x} \lesssim \|u\|_{H^{2s-1-}_x} \|v\|_{H^{1-s,\frac{3}{2s-1}+}_x} \, , $$
where $\frac{1}{q} = \half - \frac{2s-1}{3}+$ , so that $ H^{2s-1-}_x \hookrightarrow L^q_x$ . Secondly
$$ \|\Lambda^{1-s}u v\|_{L^2_x} \lesssim \|\Lambda^{1-s}u\|_{L^{q_1}_x} \|v\|_{L^{p_1}_x} \lesssim \|u\|_{H^{2s-1-}_x} \|v\|_{H^{1-s,\frac{3}{2s-1}+}_x} \, , $$
where $\frac{1}{q_1} = \half - \frac{3s-2}{3}+$ , $\frac{1}{p_1} = \frac{2s-1}{3} - \frac{1-s}{3} -$ , so that $H^{2s-1-}_x \hookrightarrow H^{1-s,q_1}_x$ and $H^{1-s,\frac{3}{2s-1}+}_x \hookrightarrow L_x^{p_1}$ . \\
{\bf a2.} We split up $A^{\mu}_{hom}$ in a low frequency part $A^{\mu}_{hom,l}$ (for frequencies $\le 1$) and a high frequency part $A^{\mu}_{hom,h}$ (for frequencies $\ge 1$). The term $A^{\mu}_{hom,h} \partial_{\mu} \phi$ can be treated like $A^{\mu}_{inh} \partial_{\mu} \phi$ in a1. We crudely estimate (for $s \le 1$) the low frequency part by use of Prop. \ref{Lemma7.1} and (\ref{10}) as follows:
\begin{align*}
&\|A^{\mu}_{hom,l} \partial_{\mu} \phi\|_{H^{\frac{7}{3}s-\frac{8}{3}-,\frac{2}{3}s-\frac{5}{6}-}}  \lesssim \|A^{\mu}_{hom,l} \partial_{\mu} \phi\|_{H^{s-\frac{4}{3}-,-\frac{1}{6}-}} \\
&\quad\lesssim  \|A^{\mu}_{hom,l} \partial_{\mu} \phi\|_{L^2_t H^{s-1,\frac{3}{2}}_x}  \lesssim \|DA^{\mu}_{hom,l}\|_{L^{\infty}_t L^2_x} \|\partial_{\mu} \phi\|_{L^{\infty}_t H^{s-1}_x} \\
& \quad\lesssim \|DA^{\mu}_{hom,l}\|_{L^{\infty}_t H^{2s-2-}_x} \|\partial_{\mu} \phi\|_{L^{\infty}_t H^{s-1}_x} < \infty \, .
\end{align*}
{\bf b.} Next we consider the cubic term $A^{\mu} A_{\mu} \phi$ . \\
{\bf b1.} We have to estimate $ \|A_{\mu}^{inh} A_{\mu}^{inh} \phi\|_{H^{\frac{7}{3}s-\frac{8}{3}-,\frac{2}{3}s-\frac{5}{6}-}}$ . By Prop. \ref{Lemma7.1} it suffices to estimate $\|A^{\mu}_{inh} A^{\mu}_{inh} \phi\|_{L^2_t H^{s-1,\frac{3}{4-2s}-}_x}$ .  By H\"older and Sobolev we obtain
$$\|A^{\mu}_{inh} A^{\mu}_{inh} \phi\|_{L^{\frac{3}{4-2s}-}_x} \lesssim  \|A^{\mu}_{inh}\|_{L^p_x}^2 \|\phi\|_{L^q_x} \lesssim\|A^{\mu}_{inh}\|_{H^{2s-1-}_x}^2 \|\phi\|_{H^s_x} \, , $$
where $\frac{4-2s}{3} + = \frac{2}{p} + \frac{1}{q} > 2(\half - \frac{2s-1}{3})+ (\half - \frac{s}{3})$ for $s > \frac{5}{6}$ , so that $p$ and $q$ can be chosen such that $H^{2s-1-}_x \hookrightarrow L_x^p$ and $H^s_x \hookrightarrow L^q_x$ . Thus
$$ \|A^{\mu}_{inh} A^{\mu}_{inh} \phi\|_{L^2_t H^{s-1,\frac{3}{4-2s}-}_x} \lesssim \|A^{\mu}_{inh} \|_{L^{\infty}_t H^{2s-1-}_x}^2 \|\phi\|_{L^{\infty}_t H^s_x} < \infty \, . $$
{\bf b2.} The same estimate holds for the term $A^{\mu}_{hom,h} A^{\mu}_{inh} \phi$ . \\
{\bf b3.} For the term $A^{\mu}_{hom,l} A^{\mu}_{inh} \phi$ we argue as in b1. and reduce to the  estimate
\begin{align*}
\|A^{\mu}_{hom,l} A^{\mu}_{inh} \phi\|_{L^{\frac{3}{4-2s}-}_x} &\lesssim \| A^{\mu}_{hom,l}\|_{L^6_x} \|A^{\mu}_{inh}\|_{L^p_x} \|\phi \|_{L^q_x} \\
&\lesssim \|DA^{\mu}_{hom,l} \|_{L^2_x} \|A^{\mu}_{inh}\|_{H^{2s-1-}_x} \|\phi\|_{H^s_x} < \infty \, , 
\end{align*}
where we need $\frac{4-2s}{3} + = \frac{1}{6}+\frac{1}{p}+\frac{1}{q}$ , and for the application of Sobolev $\half \ge \frac{1}{p} \ge \half - \frac{2s-1}{3}+$ and $\half \ge \frac{1}{q} \ge \half - \frac{s}{3}$ . This requires $\frac{4-2s}{3} > \frac{1}{6} + \frac{4}{3} -s $ , which holds for $s > \half$ , and also $\frac{4-2s}{3} < \frac{7}{6}$ , which certainly holds. \\
{\bf b4.} The term $A^{\mu}_{hom,l} A^{\mu}_{hom,h} \phi$ can be treated like b3. \\
{\bf b5.} Finally we crudely estimate
\begin{align*}
&\|A^{\mu}_{hom,l} A^{\mu}_{hom,l} \phi\|_{L^2_t H^{s-1,\frac{3}{4-2s}-}_x}  \lesssim \|A^{\mu}_{hom,l} A^{\mu}_{hom,l} \phi\|_{L^2_t L^{\frac{3}{4-2s}-}_x} \\
& \quad\lesssim \|A^{\mu}_{hom,l} \|_{L^{\infty}_t L^6_x}^2 \|\phi\|_{L^{\infty}_t L^{\frac{3}{3-2s}-}_x} \lesssim \| DA^{\mu}_{hom,l} \|_{L^{\infty}_t L^2_x}^2 \|\phi\|_{L^{\infty}_t H^s_x}
\end{align*}
by the embedding $H^s \hookrightarrow L^{\frac{3}{3-2s}-}$ , which holds for $s > \frac{3}{4}$ . \\
{\bf 2.} Next we want to prove $A^{inh}_{\mu\pm} \in X_{\pm}^{\frac{5}{3}s-1-,\half + \frac{s}{3} -}$ ,
 which requires  to estimate $ \|{\mathcal N}_{\mu} (\phi_+,\phi_-,A_+,A_-)\|_{H^{\frac{5}{3}s-2,\frac{s}{3}-\half-}} $ . \\
{\bf a.} The estimate
$$ \|\phi \overline{\partial_j \phi} \|_{H^{\frac{5}{3}s-2,\frac{s}{3}-\half-}} \lesssim \|\phi\|_{L^{\infty}_t H^s_x} \|\partial_j \phi\|_{L^{\infty}_t H^{s-1}_x} < \infty $$
follows as in part a of the proof of Theorem \ref{Theorem3}. Similarly the term $\phi \Lambda_m(\overline{\phi}_+ - \overline{\phi}_-)$ can be treated. \\
{\bf b.} Next we obtain by Prop. \ref{Lemma7.1} and Sobolev : \\
{\bf b1.}
 \begin{align*}
&\| |\phi|^2 A^{\mu}_{inh}\|_{H^{\frac{5}{3}s-2,\frac{s}{3}-\half-}} \lesssim \| |\phi|^2 A^{\mu}_{inh}\|_{L^2_t H^{s-1,\frac{3}{3-s}}_x} \lesssim \| |\phi|^2 A^{\mu}_{inh}\|_{L^2_t L^{\frac{3}{4-2s}}_x} \\
& \quad\lesssim \|\phi\|_{L^{\infty}_t L^p_x}^2 \|A^{\mu}_{inh}\|_{L^{\infty}_t L^q_x} \lesssim \|\phi\|_{L^{\infty}_t H^s_x}^2 \|A^{\mu}_{inh}\|_{L^{\infty}_t H^{2s-1-}_x} < \infty \, 
, 
\end{align*}
where $\frac{4-2s}{3} = \frac{2}{p} + \frac{1}{q}$ , $\half \ge \frac{1}{q} > \half - \frac{2s-1}{3}$ and $\half \ge \frac{1}{p} \ge \half - \frac{s}{3} $ , which requires $\frac{3}{2} \ge \frac{4-2s}{3} > \frac{11}{6} - \frac{4}{3}s $ , thus $ s > \frac{3}{4}$ . \\
{\bf b2.} The term $|\phi|^2 A^{\mu}_{hom,h}$ is treated in the same way. \\
{\bf b3.} Moreover we crudely estimate
\begin{align*}
&\| |\phi|^2 A^{\mu}_{hom,l}\|_{H^{\frac{5}{3}s-2,\frac{s}{3}-\half-}} \lesssim \| |\phi|^2 A^{\mu}_{hom,l}\|_{L^2_t L^{\frac{3}{3-s}}_x}  \\
& \quad\lesssim \|\phi\|_{L^{\infty}_t L^{\frac{6}{3-s}}_x}^2 \|A^{\mu}_{hom,l}\|_{L^{\infty}_t L^{\infty}_x} \lesssim \|\phi\|_{L^{\infty}_t H^s_x}^2 \|DA^{\mu}_{hom,l}\|_{L^{\infty}_t H^{2s-2-}_x} < \infty \, , 
\end{align*}
because $H^s_x \hookrightarrow L^{\frac{6}{3-s}}_x $ . \\
{\bf 3.} Finally we remark that $DA^{\mu}_{hom} \in H^{2s-2-,1-}$ . \\
{\bf 4.} By interpolation we obtain the following regularity of $\phi_{\pm}$ and $A^{inh}_{\mu \pm}$. \\
{\bf a.} $$\phi_{\pm} \in X^{s,0}_{\pm} \cap X_{\pm}^{\frac{7}{3}s-\frac{5}{3}-,\frac{1}{6}+\frac{2}{3}s-} \subset X^{2-s-\frac{2}{r_1},\half-\frac{1}{r_1}+}_{\pm} \, , $$
where the interpolation parameter $\theta$ is chosen such that
$ \theta(\frac{1}{6}+\frac{2}{3}s) =  \half - \frac{1}{r_1}+ $ and $(1-\theta)s+\theta(\frac{7}{3}s - \frac{5}{3}) =s-\theta(\frac{5}{3}-\frac{4}{3}s) = 2-s-\frac{2}{r_1} \, , $ so that $r_1$ is defined by
\begin{equation}
\label{r}
\frac{2}{r_1} = 2(1-s) + \frac{\half-\frac{1}{r_1}}{\frac{1}{6}+\frac{2}{3}s} \Big( \frac{5}{3}-\frac{4}{3}s \Big)+ \quad \Longleftrightarrow \quad \frac{1}{r_1} = \frac{\frac{7}{2}+s-4s^2}{6}+ \, . 
\end{equation}
This implies by Prop. \ref{Lemma7.1} : $$\phi_{\pm} \in L^2_t H^{1-s-,r_1}_x \, .$$ 
{\bf b.} In the same way we obtain
$$ A^{inh}_{\mu \pm} \in X^{2s-1-,0}_{\pm} \cap X^{\frac{5}{3}s-1,\half+\frac{s}{3}-}_{\pm} \subset X^{2-s-\frac{2}{\tilde{r}_1},\half-\frac{1}{\tilde{r}_1}+}_{\pm} \, , $$
where the interpolation parameter $\theta$ is chosen such that
$ \theta(\frac{1}{2}+\frac{1}{3}s) =  \half - \frac{1}{\tilde{r}_1}+ $ and $(1-\theta)(2s-1-) +\theta(\frac{5}{3}s - 1) =2s-1 -\theta \frac{s}{3}- = 2-s-\frac{2}{\tilde{r}_1} \, , $ so that $\tilde{r}_1$ is defined by
\begin{equation}
\label{rtilde}
\frac{2}{\tilde{r}_1} = 3(1-s) + \frac{\half-\frac{1}{\tilde{r}_1}}{\frac{1}{2}+\frac{s}{3}}+\,\cdot\frac{s}{3}  \quad \Longleftrightarrow \quad \frac{1}{\tilde{r}_1} = \frac{\frac{3}{2} - \frac{1}{3} s-s^2}{1+s}+ \, . \end{equation}
This implies by Prop. \ref{Lemma7.1} : 
\begin{equation}
\label{Amu}
A^{inh}_{\mu \pm} \in L^2_t H^{1-s,\tilde{r}_1}_x \, .
\end{equation}

This the first step of an iteration which yields more and more regularity of $A_j$ and $\phi$. The general iteration step is contained in the next two lemmas.

We define sequences iteratively by $r_1$ and $\ti{r}_1$ as given by (\ref{r}) and (\ref{rtilde}) and for $k \in {\mathbb N}$ by the following relations
\begin{align}
\label{qk'}
\frac{1}{q_k} &= \frac{1}{\ti{r}_k} + \half - \frac{1-s}{3} \\
\label{r(k+1)''}
\frac{1}{r_{k+1}} &= \frac{\frac{5}{2}-3s+\frac{1}{q_k}(2s-1)}{2} + \\
\label{B5}
\frac{1}{\ti{q}_k} &= \frac{1}{r_{k+1}} + \half - \frac{1-s}{3} \\
\label{B4}
\frac{2}{\ti{r}_{k+1}} & = \frac{\frac{7}{2}-4s+(3s-2) \frac{1}{\ti{q}_k}}{1+s}+ 
\end{align}

\begin{lemma}
\label{LemmaA}
Assume 
$$A_{\mu\pm}^{inh} \in X_{\pm}^{2-s-\frac{2}{\ti{r}_k},\half - \frac{1}{\ti{r}_k}+} \, , \, \phi_{\pm} \in  X_{\pm}^{2-s-\frac{2}{r_k},\half - \frac{1}{r_k}+} \, ,$$  
where $\half \ge \frac{1}{\ti{r}_k} > \frac{1-s}{3}$ and $\half \ge \frac{1}{r_k} > \frac{1-s}{3}$. Then we also gain
$$ \phi_{\pm} \in X_{\pm}^{s+1-\frac{2}{q_k},\frac{3}{2}-\frac{1}{q_k}-} \, . $$ 
\end{lemma}
\begin{proof}[Proof of Lemma \ref{LemmaA}:] 
Fundamental for the proof is the fact that our assumption implies by Prop. \ref{Lemma7.1} :
$ A_{j \pm}^{inh} \in L^2_t H^{1-s,\ti{r}_k}_x \, .$ \\
The claimed regularity of $\phi_{\pm}$ reduces by Prop. \ref{Lemma7.1} to proving that the right hand side of (\ref{b}) belongs to $L^2_t H^{s-1,q_k}_x$ , because  $ L^2_t H^{s-1,q_k}_x  \hookrightarrow H^{s-\frac{2}{q_k},\half-\frac{1}{q_k}-}$ . \\
{\bf a.} First we treat the term $A^{\mu} \partial_{\mu} \phi$ . \\
{\bf a1.}
The estimate
$$ \| A^{inh}_{\mu \pm} \partial_{\mu} \phi\|_{L^2_t H^{s-1,q_k}_x} \lesssim \|A^{inh}_{\mu\pm}\|_{L^2_t H^{1-s,{\ti r}_k}_x} \|\partial_{\mu} \phi\|_{L^{\infty}_t H^{s-1}_x} < \infty  $$
follows as in the proof of Lemma \ref{LemmaL1}, 1a. \\
{\bf a2.} Next we estimate similarly by using Prop. \ref{Lemma7.1} for the second step
\begin{align*}
\| A_{\mu\pm}^{hom,h} \partial_{\mu} \phi \|_{L^2_t H^{s-1,q_k}_x} &\lesssim \|A_{\mu\pm}^{hom,h} \|_{L^2_t H^{1-s,\tilde{r}_k}_x} \|\partial_{\mu} \phi\|_{L^{\infty}_t H^{s-1}_x} \\
& \lesssim \|A_{\mu\pm}^{hom,h} \|_{H^{2-s-\frac{2}{\tilde{r}_k},\half - \frac{1}{\tilde{r}_k}+}} \|\partial_{\mu} \phi\|_{L^{\infty}_t H^{s-1}_x} \\
& \lesssim \|A_{\mu\pm}^{hom,h} \|_{H^{2s-1-,1-}} \|\partial_{\mu} \phi\|_{L^{\infty}_t H^{s-1}_x} < \infty \, ,
\end{align*}
where we need $2-s-\frac{2}{\tilde{r}_k} < 2s-1 \, \Leftrightarrow \, \frac{2}{\tilde{r}_k} > 3(1-s)$ , which holds by assumption.  \\
{\bf a3.} We crudely estimate (for some sufficiently large $N$) :
\begin{align*}
\|A_{\mu\pm}^{hom,l} \partial_{\mu} \phi\|_{H^{s-\frac{2}{q_k},\half -\frac{1}{q_k}-}} & \lesssim \|A_{\mu\pm}^{hom,l} \partial_{\mu} \phi\|_{L^2_t H^{s-\frac{2}{q_k}}_x} \lesssim  \|A_{\mu\pm}^{hom,l} \partial_{\mu} \phi\|_{L^2_t H^{s-1}_x} \\
& \lesssim \|A_{\mu\pm}^{hom,l} \|_{L^{\infty}_t H^{N,6}_x} \|\partial_{\mu} \phi\|_{L^{\infty}_t H^{s-1}_x} \\ & \lesssim 
\|DA_{\mu\pm}^{hom,l} \|_{L^{\infty}_t H^{2s-2-}_x} \|\partial_{\mu} \phi\|_{L^{\infty}_t H^{s-1}_x} < \infty \, ,
\end{align*}
where we used that $\frac{1}{q_k} = \frac{1}{\tilde{r}_k} + \half - \frac{1-s}{3}  \ge \half $ by assumption.\\
{\bf b.} Now consider the term $A^{\mu \pm} A_{\mu \pm} \phi$ . \\
{\bf b1.} We start as in a1. and use Prop. \ref{Prop.4.4} for $s > \frac{3}{4}$ and (\ref{Amu}) to obtain
\begin{align*}
\|A_{\mu\pm}^{inh} A_{\mu\pm}^{inh} \phi\|_{L^2_t H^{s-1,q_k}_x} & \lesssim \|A^{\mu}_{inh}\|_{L^2_t H^{1-s,\tilde{r}_k}_x} \|A^{\mu}_{inh} \phi\|_{L^{\infty}_t H^{s-1}_x} \\
& \lesssim \|A^{\mu}_{inh}\|_{L^2_t H^{1-s,\tilde{r}_k}_x} \|A^{\mu}_{inh}\|_{L^{\infty}_t H^{2s-1-}_x} \| \phi\|_{L^{\infty}_t H^s_x} < \infty \, . 
\end{align*}
{\bf b2.} In the same way the term $A_{\mu\pm}^{inh}  A_{\mu\pm}^{hom,h} \phi$ can be treated. \\
{\bf b3.} Next, as in a1. and by use of Prop. \ref{Prop.4.4} :
\begin{align*}
\|A_{\mu\pm}^{hom,l} A_{\mu\pm}^{inh} \phi\|_{L^2_t H^{s-1,q_k}_x} & \lesssim \|A_{\mu\pm}^{inh}\|_{L^2_t H^{1-s,\tilde{r}_k}_x} \|A_{\mu\pm}^{hom,l} \phi\|_{L^{\infty}_t \dot{H}^{s-1}_x} \\& \lesssim \|A_{\mu\pm}^{hom,l}\|_{L^{\infty}_t \dot{H}^{2s-1-}_x} \|\phi\|_{L^{\infty}_t \dot{H}^{\frac{3}{2}-s+}_x} \|A_{\mu\pm}^{inh}\|_{L^2_t H^{1-s,\tilde{r}_k}_x} \\ 
& \lesssim \|DA_{\mu\pm}^{hom,l}\|_{L^{\infty}_t H^{2s-2}_x} \| \phi\|_{L^{\infty}_t H^s_x} \|A^{\mu}_{inh}\|_{L^2_t H^{1-s,\tilde{r}_k}_x} < \infty \, .
\end{align*}
{\bf b4.} We easily obtain using $q_k \le 2$
\begin{align*}
\|A_{\mu\pm}^{hom,l} A_{\mu\pm}^{hom,l} \phi\|_{H^{s-\frac{2}{q_k},\half-\frac{1}{q_k}-}} & \lesssim \|A_{\mu\pm}^{hom,l} A_{\mu\pm}^{hom,l} \phi\|_{L^2_t H^{s-1}_x} \\
& \lesssim \|A_{\mu\pm}^{hom,l} A_{\mu\pm}^{hom,l} \phi\|_{L^2_t L^2_x} \\
& \lesssim \|A_{\mu\pm}^{hom,l} \|_{L^{\infty}_t L^{\infty}_x}^2 \|\phi\|_{L^{\infty}_t L^2_x} < \infty \, .
\end{align*}
Thus we have proven $\phi_{\pm} \in X_{\pm}^{s+1-\frac{2}{q_k},\frac{3}{2}-\frac{1}{q_k}-}$ . 
\end{proof}

\begin{Cor}
\label{CorA}
Under the assumptions of Lemma \ref{LemmaA} we obtain the regularity
$\phi_{\pm} \in X_{\pm}^{2-s-\frac{2}{r_{k+1}},\half - \frac{1}{r_{k+1}}+}$ .
\end{Cor}
\begin{proof}
We interpolate the property $\phi_{\pm} \in X_{\pm}^{s+1-\frac{2}{q_k},\frac{3}{2}-\frac{1}{q_k}-} \, , $ where
$$
 \frac{1}{q_k} = \frac{1}{\ti{r}_k} + \half - \frac{1-s}{3} \, , 
$$
 with our assumption $\phi_{\pm} \in X_{\pm}^{s,0}$ and obtain  the claimed regularity, if we choose the interpolation parameter $\theta$ by the conditions $\theta(\frac{3}{2}-\frac{1}{q_k}-) = \half - \frac{1}{r_{k+1}}+ $ and $(1-\theta)s + \theta(s+1-\frac{2}{q_k}) = s - \theta(\frac{2}{q_k}-1) = 2-s - \frac{2}{r_{k+1}} $ . This requires (\ref{r(k+1)''}) .
\end{proof}

\begin{lemma}
\label{LemmaB}
Assume 
$$A_{\mu\pm}^{inh} \in X_{\pm}^{2-s-\frac{2}{\ti{r}_k},\half - \frac{1}{\ti{r}_k}+} \, , \, \phi \in  X_{\pm}^{2-s-\frac{2}{r_{k+1}},\half - \frac{1}{r_{k+1}}+} \, , $$
where $r_{k+1}$ is defined by (\ref{r(k+1)''}) and $\ti{r}_k$ by (\ref{qk'}).
Then we also gain
$$ A_{\mu\pm}^{inh} \in X_{\pm}^{s+1-\frac{2}{\ti{q}_k},\frac{3}{2}-\frac{1}{\ti{q}_k}-} \, . $$ 
\end{lemma}
\begin{proof}
{\bf a.}
The estimate
$$ \| \phi \partial_{\mu} \phi\|_{L^2_t H^{s-1,\ti{q}_k}_x} \lesssim \| \phi \|_{L^2_t H^{1-s,r_{k+1}}_x} \|\partial_{\mu} \phi\|_{L^{\infty}_t H^{s-1}_x} < \infty  $$
follows as in the proof of Lemma \ref{LemmaL1}, 1a. \\
{\bf b1.}
We want to demonstrate
$$ \| A_{\mu \pm}^{inh} |\phi|^2 \|_{L^2_t H^{s-1,\ti{q}_k}_x} \lesssim \| A_{\mu \pm}^{inh} \|_{L^2_t H^{1-s,\ti{r}_k}_x} \| \phi\|_{L^{\infty}_t H^{s}_x}^2 < \infty  \, .$$
We prove the estimate
$$ \||\phi|^2 v\|_{H^{s-1,\ti{q}_k}} \lesssim \||\phi|^2\|_{H^{s-1,\ti{t}_k}}  \|v\|_{H^{1-s,\ti{r}_k}} $$
for a suitable $\ti{t}_k$ .
By duality this is equivalent to
$$ \|vw\|_{H^{1-s,\ti{t}_k'}} \lesssim \|v\|_{H^{1-s,\ti{r}_k}} \|w\|_{H^{1-s,\ti{q}_k'}} \, . $$
By the fractional Leibniz rule we have to prove two estimates:
\begin{equation}
\label{B1}
\| (\Lambda^{1-s} v) w\|_{L^{\ti{t}_k'}}\lesssim  \|\Lambda^{1-s}v\|_{L^{\ti{r}_k}} \|w\|_{L^{p_k}} \lesssim \|\Lambda^{1-s} v\|_{L^{\ti{r}_k}} \|w\|_{H^{1-s,\ti{q}_k'}} \, .
\end{equation}
Here we need $\frac{1}{p_k} \ge \frac{1}{\ti{q}_k'} - \frac{1-s}{3}$ , so that by Sobolev $H^{1-s,\ti{q}_k'} \hookrightarrow L^{p_k}$ and $\frac{1}{\ti{t}_k'} = \frac{1}{\ti{r}_k'} + \frac{1}{p_k} \ge \frac{1}{\ti{r}_k} + \frac{1}{\ti{q}_k'} - \frac{1-s}{3}\ge \half + \frac{1}{\ti{r}_k} - \frac{1}{r_{k+1}}$ , thus we choose $\frac{1}{\ti{t}_k} = \half - \frac{1}{\ti{r}_k} + \frac{1}{r_{k+1}}$ . Similarly we also obtain
\begin{equation}
\label{B2}
\| v \Lambda^{1-s} w\|_{L^{\ti{t}_k'}} \lesssim \|v\|_{L^{\ti{p}_k}} \|\Lambda^{1-s} w\|_{L^{\ti{q}_k'}} \lesssim \|\Lambda^{1-s}v\|_{L^{\ti{r}_k}} \|\Lambda^{1-s} w\|_{L^{\ti{q}_k'}} \, , 
\end{equation}
where we need $\frac{1}{\ti{p}_k} \ge \frac{1}{\ti{r}_k} - \frac{1-s}{3}$ , so that
$\frac{1}{\ti{t}_k'} = \frac{1}{\ti{q}_k'} + \frac{1}{\ti{p}_k} \ge \frac{1}{\ti{q}_k'} + \frac{1}{\ti{r}_k} - \frac{1-s}{3}$ as above. It remains to estimate
$$\| |\phi|^2 \|_{H^{s-1,\ti{t}_k}} \lesssim \| |\phi|^2 \|_{L^{z_k}} = \| \phi \|_{L^{2z_k}}^2 \lesssim \| \phi\|_{H^s}^2 \, . $$
Here we need by Sobolev $\frac{1}{\ti{t}_k} \ge \frac{1}{z_k} - \frac{1-s}{3}$ and $ \frac{1}{2z_k} \ge \half-\frac{s}{3}$ and thus $\frac{1}{\ti{t}_k} \ge \frac{2}{3} - \frac{s}{3}$ . In view of our choice of $\ti{t}_k$  we have to prove
$ \half - \frac{1}{\ti{r}_k} + \frac{1}{r_{k+1}} \ge \frac{2}{3} - \frac{s}{3}$ . Using (\ref{r(k+1)''}) and  (\ref{qk'}) an elementary calculation proves that this is equivalent to $\frac{1}{q_k} < \frac{5}{6}$ . In the appendix 5.2 we prove that $\{\frac{1}{q_k}\}$ is monotonically decreasing, so that it remains to prove $\frac{1}{q_1}=\frac{1}{\ti{r}_1} + \half - \frac{1-s}{3} = \frac{\frac{3}{2}-\frac{1}{3}s-s^2}{1+s}+\half - \frac{1-s}{3} < \frac{5}{6}$ , using the definition of $\ti{r_1}$ by (\ref{rtilde}). One easily checks that this is equivalent to $\frac{5}{2}-2s-2s^2 < 0 $ , which holds for $s > \half(\sqrt{6} -1) \approx 0.725$ . This completes the proof of part b1. \\
{\bf b2.} Starting as in b1. we obtain
\begin{align*}
\| A_{\mu \pm}^{hom,h} |\phi|^2 \|_{L^2_t H^{s-1,\ti{q}_k}_x} &\lesssim \| A_{\mu \pm}^{hom,h} \|_{L^2_t H^{1-s,\ti{r}_k}_x} \| \phi\|_{L^{\infty}_t H^{s}_x}^2 \\
& \lesssim \|A^{hom,h}_{\mu \pm} \|_{H^{2-s-\frac{2}{\ti{r}_k},\half - \frac{1}{\ti{r}_k}+}} \| \phi\|_{L^{\infty}_t H^{s}_x}^2 \\
& \lesssim \|A^{hom,h}_{\mu \pm} \|_{H^{2s-1-,1-}} \| \phi\|_{L^{\infty}_t H^{s}_x}^2 < \infty \, ,
\end{align*}
using that  $\frac{2}{\ti{r}_k} > 3(1-s)$ by Lemma \ref{LemmaA.5}. \\
{\bf b3.} Finally we crudely estimate as follows:
\begin{align*}
\|A^{hom,l}_{\mu \pm} | \phi|^2 \|_{X_{\pm}^{s-\frac{2}{\ti{q}_k},\half-\frac{1}{\ti{q}_k}-}} & \lesssim \|A^{hom,l}_{\mu \pm} | \phi|^2 \|_{L^2_t H_x^{s-\frac{2}{\ti{q}_k}}} \\
& \lesssim \|A^{hom,l}_{\mu \pm} | \phi|^2 \|_{L^2_t H^{s-1}_x} \\
& \lesssim \|A^{hom,l}_{\mu \pm} \|_{L^{\infty}_t H^{N,6}_x} \| |\phi|^2 \|_{L^{\infty}_t H^{s-1}_x} \\
& \lesssim \|D A^{hom,l}_{\mu \pm} \|_{L^{\infty}_t H^{2s-2-}_x} \| \phi \|_{L^{\infty}_t H^s_x}^2
\end{align*}
with a sufficiently large $N$.
\end{proof}

\begin{Cor}
\label{Cor.B}
Under the assumptions of Lemma \ref{LemmaB} we obtain the regularity $A_{\mu \pm}^{inh} \in X_{\pm}^{2-s-\frac{2}{\ti{r}_{k+1}}, \half - \frac{1}{\ti{r}_{k+1}}+}$ .
\end{Cor}
\begin{proof} 
By interpolation we obtain the following regularity
$ A_{\mu \pm}^{inh} \in X_{\pm}^{s+1-\frac{2}{\ti{q}_k},\frac{3}{2}-\frac{1}{\ti{q}_k}-} \\\cap X_{\pm}^{2s-1-,0} 
\subset X_{\pm}^{2-s-\frac{2}{\ti{r}_{k+1}},\half - \frac{1}{\ti{r}_{k+1}}+} \, , $ where the interpolation parameter $\theta$ is chosen such $\theta(\frac{3}{2}-\frac{1}{\ti{q}_k}-) = \half - \frac{1}{\ti{r}_{k+1}}+ $ and $\theta(s+1-\frac{2}{\ti{q}_k}) + (1-\theta)(2s-1-) = 2-s-\frac{2}{\ti{r}_{k+1}}$. One calculates that this leads to (\ref{B4}).
\end{proof}
By Lemma \ref{LemmaA.5} the assumptions of 
Lemma \ref{LemmaA} and Lemma \ref{LemmaB} are satisfied, so that iteratively we conclude
$$ \phi_{\pm} \in X_{\pm}^{s+1-\frac{2}{q_k},\frac{3}{2}-\frac{1}{q_k}-} \quad \forall \,k\in \mathbb N $$
and 
$$ A^{inh}_{\mu \pm} \in X_{\pm}^{s+1-\frac{2}{\ti{q}_k},\frac{3}{2}-\frac{1}{\ti{q}_k}-} \quad \forall \,k \in \mathbb N \, . $$

In the appendix 5.2 we prove $q_k \to q_{\infty}$ , $r_k \to r_{\infty}$ , $\ti{q}_k \to \ti{q}_{\infty}$ , $\ti{r}_k \to \ti{r}_{\infty}$ and determine these limits explicitly. Assuming this we immediately conclude
$$ \phi_{\pm} \in X_{\pm}^{s+1-\frac{2}{q_\infty}-,\frac{3}{2}-\frac{1}{q_{\infty}}-} 
 \, , \,  A^{inh}_{\mu \pm} \in X_{\pm}^{s+1-\frac{2}{\ti{q}_{\infty}}-,\frac{3}{2}-\frac{1}{\ti{q}_{\infty}}-}  $$

Consequently we obtain
\begin{lemma}
\label{LemmaC} 
$\phi_{\pm} \in X_{\pm}^{s_{\infty}-,\frac{1}{2}+} \, , $
where $s_{\infty} = s - \frac{\frac{1}{2}(\frac{2}{q_{\infty}}-1)}{\frac{3}{2}-\frac{1}{q_{\infty}}} -$ .
\end{lemma}
\begin{proof}
This follows by interpolation. We obtain
$$ \phi_{\pm} \in X_{\pm}^{s+1-\frac{2}{q_{\infty}}-,\frac{3}{2}-\frac{1}{q_{\infty}}-} \cap X_{\pm}^{s,0} \subset X_{\pm}^{s_{\infty}-,\half+} \, , $$
if we choose the interpolation parameter  $\theta$ such that $\theta(\frac{3}{2} - \frac{1}{q_{\infty}}-) = \half+$ , so that 
\begin{equation}
\label{sinf'}
s_{\infty} = (1-\theta)s + \theta(s+1-\frac{2}{q_{\infty}}) = s-\theta(\frac{2}{q_{\infty}}-1) =  s - \frac{\half(\frac{2}{q_{\infty}}-1)}{\frac{3}{2}-\frac{1}{q_{\infty}}} -\, . 
\end{equation}
\end{proof}
Now we calculate
\begin{align*} 
s_{\infty} > \frac{3}{4} & \Longleftrightarrow (4s-3)(\frac{3}{2}-\frac{1}{q_{\infty}}) - 2(\frac{2}{q_{\infty}}-1) > 0 \Longleftrightarrow \frac{1}{q_{\infty}} < \frac{\frac{3}{2}s-\frac{5}{8}}{s+\frac{1}{4}} \, . 
\end{align*}
Using (\ref{qinf}) in the appendix 5.2 this results in the following condition for $s$ :
$$\frac{\frac{5}{6}+\frac{37}{12}s-\frac{19}{6}s^2}{s(\frac{9}{2} -3s)} <  \frac{\frac{3}{2}s-\frac{5}{8}}{s+\frac{1}{4}} \Longleftrightarrow 32s^3-152s^2+106s+5 > 0 \, . $$
Thus  this is fufilled
 for $s > s_0$ , where $s_0$ denotes the largest root of this cubic polynomial. This coincides  with our assumption on $s$. A calculation yields $0,906 < s_0 < 0.907$ . 

It remains to prove $A^{inh}_{\mu \pm} \in X^{\half+,1-}_{\pm}$ . \\
{\bf a.} By Theorem \ref{Theorem3.2} we easily obtain
$$ \|\phi \partial_{\mu} \phi \|_{H^{-\half+,0}} \lesssim \|\phi\|_{H^{\frac{3}{4}+\half+}} \|\partial_{\mu} \phi\|_{H^{-\frac{1}{4}+,\half+}} < \infty \, . $$
{\bf b1.} Next we obtain
\begin{align*}
\| |\phi|^2 A^{inh}_{\mu} \|_{H^{-\half+,0}} &\lesssim \| |\phi|^2 \|_{H^{2-2s+,0}} \|A^{inh}_{\mu}\|_{L^{\infty}_t H^{2s-1-}_x} \\
& \lesssim  \|\phi\|_{H^{\frac{3}{4}+,\half+}}^2\|A^{inh}_{\mu}\|_{L^{\infty}_t H^{2s-1-}_x} \, , 
\end{align*}
where we used Prop. \ref{Prop.4.4} for the first step and Theorem \ref{Theorem3.2} for the second step, which holds for $s > \frac{3}{4}$ . \\
{\bf b2.} The same estimate holds for $|\phi|^2 A^{\mu}_{hom,l}$ . \\
{\bf b3.} Finally we obtain for $s > \frac{3}{4}$ :
\begin{align*}
\| |\phi|^2 A^{\mu}_{hom,l} \|_{H^{-\half+,0}} & \lesssim \|  |\phi|^2 A^{\mu}_{hom,l} \|_{L^2_t L^{\frac{3}{2}+}_x} \lesssim \| A^{\mu}_{hom,l} \|_{L^{\infty}_t L^6_x} \|\phi\|_{L^{\infty}_t L^{4+}_x}^2 \\
& \lesssim \| D A^{\mu}_{hom,l} \|_{L^{\infty}_t H^{2s-2-}_x} \|\phi\|^2_{L^{\infty}_t H^s_x} < \infty \, .
\end{align*}

The proof of Theorem \ref{Theorem5} is complete.
\end{proof}

\section{Appendix}
\noindent{\bf 5.1 Coulomb gauge.}
We define for $k \in \mathbb{N}$ : $t_k := \frac{1}{r_k}$ , $t_{\infty} := \frac{1}{r_{\infty}}$ , where $r_k$ is defined by (\ref{r1}) and (\ref{r(k+1)}), and $r_{\infty}$ by (\ref{rinf}).
\begin{lemma}
\label{LemmaA.1}
If $1 \ge s \ge \frac{3}{4}$ we obtain $t_k > t_{\infty}$ $\forall \, k \in {\mathbb N}$ .
\end{lemma}
\begin{proof} (by induction) 
Using (\ref{r1}) and (\ref{rinf}) we first prove 
\begin{equation}
\label{IA}
t_1 = \frac{3}{4} - \frac{1}{3}s-\frac{1}{3}s^2 -> t_{\infty} = \frac{(\frac{7}{3}-\frac{2}{3}s)(1-s)}{3-2s}+ \, . 
\end{equation}
This is equivalent to
$$
(3-2s) (\frac{3}{4} - \frac{1}{3}s-\frac{1}{3}s^2)  > (\frac{7}{3}-\frac{2}{3}s)(1-s) \quad
\Leftrightarrow \quad s(\frac{2}{3}s^2 - s + \half) - \frac{1}{12}  > 0 \, .
$$
If $s \ge \frac{3}{4}$ , this is a consequence of
$$ \frac{3}{4}(\frac{2}{3}s^2 - s + \half) - \frac{1}{12}  > 0 \quad \Leftrightarrow \frac{2}{3}s^2 -s+\half-\frac{1}{9} > 0 \, ,
$$
because $$\frac{2}{3}s^2-s+\half > 0 \, \Leftrightarrow \, (s-\frac{3}{4})^2 + \frac{3}{16} > 0 \,,$$ which is true. Now one easily checks
$$  \frac{2}{3}s^2 -s+\half-\frac{1}{9} > 0 \quad \Leftrightarrow (s-\frac{3}{4})^2 + \frac{1}{48} > 0 \, ,$$
which certainly holds. This implies (\ref{IA}).

For the iteration step we start by reformulating (\ref{r(k+1)}) as follows:
\begin{align}
\nonumber
2t_{k+1}(1-t_k + \frac{1-s}{3}) &= 2(1-s)(1-t_k + \frac{1-s}{3}) + (\half-t_{k+1})(2t_k - \frac{2(1-s)}{3}+ \\
\label{t(k+1)}
\Longleftrightarrow \quad
t_{k+1} & = (1-s)(\frac{5}{6}-t_k + \frac{1}{3}(1-s)) + \half t_k +\, ,
\end{align}
We want to prove that $t_{k+1} > t_{\infty}$ provided $t_k > t_{\infty}$ . Now by (\ref{t(k+1)}) we have
$$ t_{k+1} > t_{\infty} \quad \Longleftrightarrow \quad (\half-(1-s)) t_k + (1-s)(\frac{5}{6}+\frac{1}{3}(1-s)) + > t_{\infty} \, . $$
Using $\half-(1-s) > 0$ for $s > \half$ and $t_k > t_{\infty}$ this reduces to
$$(\half-(1-s))t_{\infty} +(1-s)(\frac{5}{6}+\frac{1}{3}(1-s))+ \ge t_{\infty} \, ,$$
 which holds by equality.
\end{proof}

\begin{lemma}
\label{LemmaA.2}
If $ 1 \ge s \ge \frac{3}{4} $ , we have $t_{k+1} \le t_k$ $\forall \, k \in {\mathbb N}$ .
\end{lemma}
\begin{proof}
By (\ref{t(k+1)}) and (\ref{rinf}) we obtain
\begin{align*}
t_{k+1} \le t_k &\Longleftrightarrow (1-s)(\frac{5}{6}-t_k + \frac{1}{3}(1-s)) + \frac{t_k}{2} + \le t_k \\&\Longleftrightarrow t_k \ge \frac{(1-s)(\frac{7}{3}-\frac{2}{3}s)}{3-2s} + = t_{\infty} \, .
\end{align*}
This holds by Lemma \ref{LemmaA.1}.
\end{proof}

\begin{Cor}
If $1 \ge s \ge \frac{3}{4}$ we obtain
$t_k \to t_{\infty}$ and $\frac{1}{r_k} \to \frac{1}{r_{\infty}}$ as $k \to \infty$ .
\end{Cor}
\begin{proof}
This follows immediately by combining Lemma \ref{LemmaA.1} and Lemma \ref{LemmaA.2}.
\end{proof} 
\noindent
{\bf Remark 5.1.} By Lemma \ref{LemmaA.1} we know for $1 \ge s \ge \frac{3}{4}$ : \\
$  \frac{1}{r_k} \ge \frac{1}{r_{\infty}} = \frac{(1-s)(\frac{7}{3}-\frac{2}{3}s)}{3-2s} + \ge \frac{1-s}{3} \quad$
and
$\quad\frac{1}{r_k} \le \frac{1}{r_1} = \frac{3}{4}-\frac{1}{3}s - \frac{1}{3}s^2 -\le \half \, . $\\[1em]

\noindent{\bf 5.2. Lorenz gauge.}
 We define
\begin{equation}
\label{rinf'}
 \frac{1}{r_{\infty}} := \frac{(\frac{7}{3}-3s+\frac{2}{3}s^2)(1+s) + (2s-1)(\frac{19}{6}-\frac{25}{6}s+s^2)}{2(1+s) - (3s-2)(2s-1)} 
\end{equation}
or equivalently
\begin{equation}
\label{rinf''}
\frac{1}{r_{\infty}} 
= \frac{\frac{7}{3}-3s+\frac{2}{3}s^2 + (2s-1)\frac{\frac{19}{6}-\frac{25}{6}s + s^2 + (3s-2)\frac{1}{r_\infty}}{1+s}}{2}
\end{equation} 
and want to prove that for $ 1 \ge s>s_0$ we obtain $\frac{1}{r_k} \to \frac{1}{r_{\infty}}$ as $k \to \infty$  where $r_k$ is iteratively defined in the following way. Using (\ref{B4}) and (\ref{B5}) we obtain
\begin{align}
\label{B6}
\frac{1}{\ti{r}_{k+1}}& = \frac{3(1-s)(\frac{3}{2} - \frac{1}{\ti{q}_k}) + \half(s-2+\frac{2}{\ti{q}_k})}{1+s} = \frac{\frac{7}{2}-4s+(\frac{1}{r_{k+1}}+\half-\frac{1-s}{3})(3s-2)}{1+s} \\
&= \frac{\frac{19}{6}-\frac{25}{6}s+s^2+(3s-2)\frac{1}{r_{k+1}}}{1+s} \, .
\end{align}
Next we use (\ref{r(k+1)''}), (\ref{qk'}) and obtain for $k \ge 1$ :
\begin{align}
\nonumber
\frac{1}{r_{k+1}} & = \frac{\frac{5}{2}-3s+\frac{1}{q_k}(2s-1)}{2} = \frac{\frac{5}{2}-3s+(\frac{1}{\ti{r}_k} + \half - \frac{1-s}{3})(2s-1)}{2} \\
\label{B7}
& = \frac{\frac{7}{3}-3s+\frac{2}{3} s^2 + (2s-1)\frac{1}{\ti{r}_k}}{2} \, ,
\end{align}
and by (\ref{B6}) this implies for $k \ge 2$ :
\begin{align}
\label{B8}
\frac{1}{r_{k+1}} &
= \frac{\frac{7}{3}-3s+\frac{2}{3}s^2 + (2s-1)\frac{\frac{19}{6}-\frac{25}{6}s + s^2 + (3s-2)\frac{1}{r_k}}{1+s}}{2} \end{align}
This together with (\ref{r})  $ \frac{1}{r_1} = \frac{\frac{7}{2}+s-4s^2}{6}+$ , and (\ref{rtilde}) $\frac{1}{\ti{r}_1} = \frac{\frac{3}{2}-\frac{1}{3}s-s^2}{1+s}+ $ , defines the sequence $\frac{1}{r_k}$ . This implies $\frac{1}{r_k} \to \frac{1}{r_{\infty}}$ , provided convergence holds. 
\begin{lemma}
\label{LemmaA.3}
If $1 \ge s > \frac{9}{10}$ we obtain $\frac{1}{r_k} > \frac{1}{r_{\infty}}$ $\forall \, k \in {\mathbb N}$ .
\end{lemma}
\begin{proof}
(by induction) 
The case $k=1$ can be established by elementary calculations using the definition of $r_1$ and $r_{\infty}$ by (\ref{r}) and (\ref{rinf'}) , or alternatively by Lemma \ref{LemmaA.4}.
By (\ref{rinf''}),(\ref{B7}) and (\ref{rtilde}) we obtain 
\begin{align*}
\frac{1}{r_2} > \frac{1}{r_{\infty}} & \Longleftrightarrow \frac{3}{2}-\frac{1}{3}s-s^2 > \frac{19}{6} - \frac{25}{6}s+s^2+(3s-2) \frac{1}{r_{\infty}} \, .
\end{align*}
Inserting $r_{\infty}$ by (\ref{rinf'}) an elementary calculation yields that this is fulfilled at least for $1 \ge s > \frac{9}{10}$ .
Next we prove$\frac{1}{r_{k+1}} > \frac{1}{r_{\infty}}$, if $\frac{1}{r_k} > \frac{1}{r_{\infty}}$ for some $k \ge 2$. This assumption implies that the right hand side (\ref{B8}) strictly decreases for $s \ge \frac{2}{3}$ , if we replace $\frac{1}{r_k}$ by $\frac{1}{r_\infty}$ , so that it equals  $\frac{1}{r_\infty}$ by (\ref{rinf''}), thus $\frac{1}{r_{k+1}} > \frac{1}{r_{\infty}}$ .
\end{proof}

\begin{lemma}
\label{LemmaA.4}
If $1 \ge s > \frac{9}{10}$ we obtain $\frac{1}{r_{k+1}} \le \frac{1}{r_k}$ $\forall \, k \in {\mathbb N}$ .
\end{lemma}
\begin{proof}
In the case $k=1$ we use (\ref{B7}) and the definition of $r_1$ and $\ti{r}_1$ by (\ref{r}) and (\ref{rtilde}). By easy calculations the inequality $\frac{1}{r_2} \le \frac{1}{r_1}$ reduces to $-3s^2+ \frac{7}{2}s-1 \le 0$ , which holds for $s \ge \frac{2}{3}$ .
In the case $k \ge 2$ we use Lemma \ref{LemmaA.3}. By (\ref{B8}) we have to prove
$$
\frac{\frac{7}{3}-3s+\frac{2}{3}s^2 + (2s-1)\frac{\frac{19}{6}-\frac{25}{6}s + s^2 + (3s-2)\frac{1}{r_k}}{1+s}}{2} \le \frac{1}{r_k} \, , $$
which can be proven to be equivalent to $\frac{1}{r_k} \ge \frac{1}{r_\infty}$ similarly to the equivalence of (\ref{rinf''}) and (\ref{rinf'}).
\end{proof}

These lemmata imply that $\frac{1}{r_k}$ is a monotonically decreasing bounded sequence, so that we now have proven that $\frac{1}{r_k} \to \frac{1}{r_{\infty}}$ as $k \to \infty$ .

We recall the following definitions:
\begin{align*}
\frac{1}{q_{k}} & = \frac{1}{\ti{r}_{k}} + \half - \frac{1-s}{3} & (\ref{qk'}) \\
\frac{1}{r_{k+1}} & = \frac{\frac{5}{2}-3s+(2s-1) \frac{1}{q_k}}{2} & (\ref{r(k+1)''}) \\
\frac{1}{\ti{q}_k} & = \frac{1}{r_{k+1}} + \half - \frac{1-s}{3} & (\ref{B5}) \\
\frac{1}{\ti{r}_{k+1}} & = \frac{\frac{7}{2}-4s+(3s-2) \frac{1}{\ti{q}_k}}{1+s} & (\ref{B4}) 
\end{align*}

Consequently we also obtain  $\frac{1}{\ti{q}_k} \to \frac{1}{\ti{q}_\infty}$ , $\frac{1}{\ti{r}_k} \to \frac{1}{\ti{r}_{\infty}}$ , $\frac{1}{q_k} \to \frac{1}{q_{\infty}}$ as $k \to \infty$  and these sequences are monotonically decreasing.

We are especially interested in the value of $q_{\infty}$, which we now determine by use of (\ref{qk'}),(\ref{r(k+1)''}),(\ref{B5}) and(\ref{B4}). We obtain
\begin{align*}
\frac{1}{q_{k+1}} & = \frac{\frac{7}{2}-4s+(3s-2)(\frac{1}{r_{k+1}} + \half - \frac{1-s}{3})}{1+s} + \half - \frac{1-s}{3} \\
& = \frac{\frac{7}{2}-4s+(3s-2)(\frac{5}{4}-\frac{3}{2}s+(2s-1)\frac{1}{2q_k}+\half-\frac{1-s}{3})}{1+s} + \half - \frac{1-s}{3} \, ,
\end{align*}
which implies
\begin{align*}
\frac{1}{q_{\infty}} 
& = \frac{\frac{7}{2}-4s+(3s-2)(\frac{5}{4}-\frac{3}{2}s+(2s-1)\frac{1}{2q_{\infty}}+\half-\frac{1-s}{3})}{1+s} + \half - \frac{1-s}{3} 
\end{align*}
and by an elementary calculation finally
\begin{equation}
\label{qinf}
\frac{1}{q_{\infty}} = \frac{\frac{5}{6}+\frac{37}{12}s-\frac{19}{6}s^2}{s(\frac{9}{2}-3s)} \, .
\end{equation}

\begin{lemma}
\label{LemmaA.5} For $ 1 \ge s \ge \frac{3}{4}$ the following estimates hold:
$ \half > \frac{1}{\ti{r}_k} > \frac{3}{2}(1-s)$ and $\half > \frac{1}{r_k} > 1-s \quad\forall k \in {\mathbb N}$ .
\end{lemma}
\begin{proof}
Because the sequences are monotonically decreasing we only have to prove $\frac{1}{r_1} < \half$ , $\frac{1}{\ti{r}_1} < \half$ , $\frac{1}{r_{\infty}} > 1-s$ and$\frac{1}{\ti{r}_{\infty}} > \frac{3}{2}(1-s)$ . The first two estimates are easily checked for $s \ge \frac{3}{4}$ by the definition of $\frac{1}{r_1}$ and $\frac{1}{\ti{r}_1}$ in (\ref{r}) and (\ref{rtilde}). By (\ref{qinf}) we have $\frac{1}{q_{\infty}} > \half$ and obtain by (\ref{r(k+1)''}) the estimate
$$ \frac{1}{r_{\infty}} = \frac{\frac{5}{2}-3s+(2s-1)\frac{1}{q_{\infty}}}{2} > \frac{\frac{5}{2}-3s+(2s-1)\half}{2} = 1-s \, . $$
Using the last estimate , (\ref{B4}) and (\ref{B5}) we obtain
\begin{align*}
\frac{1}{\ti{r}_{\infty}} & = \frac{\frac{7}{2}-4s+(3s-2)\frac{1}{\ti{q}_{\infty}}}{1+s} = \frac{\frac{7}{2}-4s+(3s-2)\frac{1}{r_{\infty}} + (3s-2)(\half - \frac{1-s}{3})}{1+s} \\
& = \frac{\frac{19}{6} - \frac{25}{6}s + s^2 + (3s-2)\frac{1}{r_{\infty}}}{1+s} > \frac{\frac{19}{6}-\frac{25}{6}s+s^2 + (3s-2)(1-s)}{1+s} \, .
\end{align*}
One easily checks that this term is bounded below by $\frac{3}{2}(1-s)$ as required, if $ (s-\frac{5}{6})^2 - \frac{1}{36} \le 0$ ,  which holds for $\frac{2}{3} \le s \le 1$ .
\end{proof}

\end{document}